\documentclass[10pt]{article}
\usepackage{amsmath}
\numberwithin{equation}{section}
\usepackage{amsfonts}
\usepackage{amssymb}
\usepackage{graphicx}
\usepackage{mathrsfs}
\usepackage{xcolor}
\usepackage{verbatim}
\usepackage{mathrsfs}
\usepackage[body={15.5cm,21cm}, top=3cm]{geometry}
\usepackage{paralist}
\usepackage{ntheorem}
\usepackage{appendix}

\allowdisplaybreaks[4]
\usepackage{hyperref}
\hypersetup{colorlinks=true,
linkcolor=blue,
anchorcolor=blue,
citecolor=blue}
\providecommand{\U}[1]{\protect \rule{.1in}{.1in}}
\newtheorem{theorem}{Theorem}[section]

\newtheorem{corollary}[theorem]{Corollary}

\newtheorem{definition}[theorem]{Definition}

\newtheorem{lemma}[theorem]{Lemma}

\newtheorem{remark}[theorem]{Remark}

\theoremstyle{empty}

\newenvironment{proof}[1][Proof]{\noindent \textbf{#1.} }{\  \rule{0.5em}{0.5em}}
\allowdisplaybreaks[2]
\begin{document}
\title{Turnpike properties for stochastic backward linear-quadratic optimal problems}
\author{Yuyang Chen\thanks{School of Mathematical Sciences, Shanghai Jiao Tong University, China (cyy0032@sjtu.edu.cn)}
\and
Peng Luo \thanks{School of Mathematical Sciences, Shanghai Jiao Tong University, China (peng.luo@sjtu.edu.cn). Financial support from the National Natural Science Foundation of China (Grant No. 12101400) is gratefully acknowledged.}}

\maketitle
\begin{abstract}
This paper deals with the long time behavior of the optimal solution of stochastic backward linear-quadratic optimal control problem over the finite time horizon. Both weak and strong turnpike properties are established under appropriate conditions, including stabilizability condition. The key ingredients are to formulate the corresponding  static optimization problem and determine the correction processes. However, our techniques are quite different from stochastic (forward) linear-quadratic case.
\end{abstract}


\textbf{MSC-classification}: 49N10, 93D23, 93E15, 93E20.

\section{Introduction}
Let $(\Omega,\mathscr{F},\mathbb{P})$ be a complete probability space on which a standard one-dimensional Brownian motion $W=\{W(t);t\geqslant0\}$ is defined and denote by $\mathbb{F}=\left\{\mathscr{F}_t\right\}_{t \geqslant 0}$ the usual augmentation of the natural filtration generated by $W$. Consider the following controlled linear backward stochastic differential equations (BSDE, for short)
\begin{equation}\label{BSDE}
\left\{\begin{array}{l}
dY(t)=[AY(t)+Bu(t)+CZ(t)+b]dt+Z(t)dW_t,\quad t\in[0,T],\\
Y(T)=\xi,
\end{array}\right.
\end{equation}
and the following general quadratic cost functional
\begin{equation}\label{cost}
\begin{aligned}
J_{T}(\xi;u(\cdot))=\frac{1}{2}\mathbb{E}&\int_{0}^{T}[\langle QY(t),Y(t)\rangle+\langle NZ(t),Z(t)\rangle+\langle Ru(t), u(t)\rangle\\
&+2\langle q,Y(t)\rangle+2\langle \mathbf{n},Z(t)\rangle+2\langle r,u(t)\rangle]dt
\end{aligned}
\end{equation}
where $A,C\in\mathbb{R}^{n\times n},~B\in\mathbb{R}^{n\times m},~Q,N\in\mathbb{S}^{n},~R\in\mathbb{S}^{m},~b,q,\mathbf{n}\in\mathbb{R}^{n}$, and $r\in\mathbb{R}^{m}$ are constant matrices or vectors with $\mathbb{S}^{n}$ being the set of all $(n\times n)$ symmetric matrices, and $\xi\in L^{\infty}_{\mathcal{F}_{T}}(\Omega,\mathbb{R}^{n})$, i.e. $\xi\in\mathcal{F}_{T}$ and is bounded. The classical stochastic backward linear-quadratic (BLQ, for short) optimal control problem over the finite time horizon $[0,T]$ is to find a control $\bar{u}_{T}(\cdot)$ from the space
\begin{equation}\label{U}
\mathscr{U}_{m}[0,T]=\left\{u:[0,T]\times\Omega\rightarrow\mathbb{R}^{m}|u\text{ is }\mathbb{F}\text{-progressively measurable and }\mathbb{E}\int_{0}^{T}|u(t)|^{2} d t<\infty\right\}
\end{equation}
such that the cost functional \eqref{cost} is minimized over $\mathscr{U}_{m}[0,T]$, for a given terminal state $\xi\in L^{\infty}_{\mathcal{F}_{T}}(\Omega,\mathbb{R}^{n})$. More precisely, it can be stated as follows.\\
\\
$\mathbf{Problem~(BLQ)_{T}}$. For any given terminal state $\xi\in L^{\infty}_{\mathcal{F}_{T}}(\Omega,\mathbb{R}^{n})$, find a control $\bar{u}_{T}(\cdot)\in\mathscr{U}_{m}[0,T]$ such that
\begin{equation}\label{BLQ}
J_{T}\left(\xi;\bar{u}_{T}(\cdot)\right)=\inf_{u(\cdot)\in\mathscr{U}_{m}[0,T]}J_{T}(\xi;u(\cdot))\equiv V_{T}(\xi).
\end{equation}
The process $\bar{u}_{T}(\cdot)$ in \eqref{BLQ} (if exists) is called an open-loop optimal control of Problem (BLQ)$_{T}$ for the terminal state $\xi$, the corresponding state process $\left(\bar{Y}_{T}(\cdot),\bar{Z}_{T}(\cdot)\right)$ is called an open-loop optimal state process, $\left(\bar{Y}_{T}(\cdot),\bar{Z}_{T}(\cdot),\bar{u}_{T}(\cdot)\right)$ is called an open-loop optimal pair, and $V_{T}(\cdot)$ is called the value function of Problem (BLQ)$_{T}$.

The stochastic BLQ optimal control problem was first proposed by Dokuchaev and Zhou \cite{Dokuchaev and Zhou 1999}. By applying the completion-of-squares method and the decoupling method, Lim and Zhou \cite{Lim and Zhou 2001} solved the stochastic BLQ problem with deterministic coefficients and obtained the optimal control represented by the adjoint process. Since then, the stochastic BLQ optimal control problem has attracted much interest. For example, Zhang \cite{Zhang 2011} discussed the stochastic BLQ optimal control problem with random jumps; Wang et al. \cite{Wang et al. 2012} solved the stochastic BLQ optimal control problem under partial information; Du et al. \cite{Du et al. 2018} and Li et al. \cite{Li et al. 2019} studied stochastic mean-field BLQ optimal control problem; Bi et al. \cite{Bi et al. 2020} obtained a theory of optimal control for controllable stochastic linear systems; Sun et al. studied indefinite stochastic BLQ optimal control problem in \cite{Sun et al. 2021}. Recently, Sun and Wang \cite{Sun and Wang 2021} solved the stochastic BLQ optimal control problem with random coefficients, which is more difficult than the forward case. We also refer readers to the wide application of stochastic BLQ optimal control problems, such as the quadratic hedging problem \cite{Lim and Zhou 2001}, the non-zero sum differential game \cite{Zhang 2011, Wang  et al. 2018}, the pension fund optimization problem \cite{Huang et al. 2009}, the Sackelberg differential game \cite{Feng et al. 2022}, etc.

For the stochastic BLQ optimal control problem in finite time horizon, one can obtain its optimal control by constructing the FBSDE with optimal condition and represent its optimal pair with adjoint process, differential Riccati equation and related BSDE. However, to the best of our knowledge, there is no literature on the long time behavior of the optimal pair of Problem(BLQ)$_{T}$, which is the main aim of this paper. More precisely, we want to find out if there are positive constants $K,\mu$ independent of $T$, such that it holds
\begin{equation}\label{weak}
\left|\mathbb{E}\left[\bar{Y}_{T}(t)-y^{*}\right]\right|+\left|\mathbb{E}\left[\bar{u}_{T}(t)-u^{*}\right]\right| \leqslant K\left[e^{-\mu t}+e^{-\mu(T-t)}\right], \quad \forall t\in[0,T]
\end{equation}
for $(y^{*},z^{*},u^{*})\in\mathbb{R}^{n}\times\mathbb{R}^{n}\times\mathbb{R}^{m}$, which is the minimizer of 
\begin{equation}\label{static}
F(y,z,u)=\frac{1}{2}[\langle Qy,y\rangle+\langle Nz,z\rangle+\langle Ru,u\rangle+2\langle q,y\rangle+2\langle \mathbf{n},z\rangle+2\langle r,u\rangle+\langle \Sigma^{-1}z,z\rangle]
\end{equation}
under $Ay+Bu+Cz+b=0$, where $\Sigma$ is a positive definite solution to the following algebraic Riccati equation (ARE for short)
$$
\Sigma A^{\top}+A\Sigma+\Sigma Q\Sigma-BR^{-1}B^{\top}-C(I_{n}+\Sigma N)^{-1}\Sigma C^{\top}=0.
$$
This is called the \textbf{(weak) turnpike property} of Problem(BLQ)$_{T}$. 

The study of the turnpike phenomenon can be traced back to \cite{Neumann 1945}, arising from the problem of optimal economic growth. In 1958, Dorfman et al. \cite{Dorfman et al. 1958} coined the name ``turnpike". This phenomenon has been studied for a long time because it originates not only in economics \cite{Makarov and Rubinov 2012,McKenzie 1976,Rubinov 1984}, but also in engineering \cite{Zaslavski 2005,Zaslavski and Leizarowitz 1998}, game theory \cite{Kolokoltsov and Yang 2012,Zaslavski 2011}, thermodynamic equilibrium theory of materials \cite{Marcus and Zaslavski 2002,Coleman et al. 1991} and so on. Some recent results can be found in \cite{Breiten and Pfeiffer 2020,Faulwasser and Grune 2022,Gugat and Lazar 2023}. Nevertheless, all these results were concerned with turnpike properties for deterministic optimal problems whose state equations were ODEs. In 2022, Sun et al. \cite{Sun et al. 2022} creatively proposed turnpike properties for stochastic linear-quadratic optimal control problems, which were totally different from the deterministic ones. Soon after, Sun and Yong \cite{Sun and Yong 2022} extended this property to the mean-field stochastic linear-quadratic optimal control problems and proved the so-called strong turnpike property. In our case, the \textbf{strong turnpike property} means 
\begin{equation}\label{strong}
\mathbb{E}\left[|\bar{Y}_{T}(t)-\bar{Y}^{*}_{T}(t)|^{2}\right]+\mathbb{E}\left[|\bar{u}_{T}(t)-\bar{u}^{*}_{T}(t)|^{2}\right]\leqslant K\left[e^{-\zeta t}+e^{-\zeta(T-t)}\right], \quad \forall t\in[0,T]
\end{equation}
holds for some constants $K,\zeta>0$ independent of $T$ and processes $\bar{Y}^{*}_{T}(\cdot),\bar{u}^{*}_{T}(\cdot)$ determined explicitly.

In this paper, we study both weak and strong turnpike properties for the stochastic BLQ optimal control problem in the finite time horizon. We should emphasize that in \cite{Sun et al. 2022} and \cite{Sun and Yong 2022}, their results were based on the stability for controlled SDEs in an infinite time horizon. Unfortunately, there are no existing stability results for controlled BSDEs in the infinite time horizon. Therefore, we need to develop a different approach to deal with stochastic BLQ optimal control setting. To start with, instead of dealing directly with the algebraic Riccati equation associated with the stochastic BLQ optimal control problem, we consider the inverse of this algebraic Riccati equation. The key observation is that the inverse is associated with a forward LQ problem, for which, we can introduce the stability of the related controlled SDEs. We further consider a family of differential Riccati equations. We establish the well-posedness and study the long time behavior. In particular, we show that the difference between the solutions of the differential Riccati equation and related algebraic Riccati equation decays exponentially by imposing stabilizability condition for related controlled SDE. Relying on this exponential decay property, we could establish both weak and strong turnpike properties for the stochastic BLQ optimal control problem. The key ingredients are to find the correct formulation of the corresponding static optimization problem and determine the correction processes. However, due to the randomness of the terminal state which is critical for stochastic BLQ, the related adjoint equation is indeed BSDE instead of ODE, which brings additional difficulty. We overcome this difficulty by some delicate estimates based on BSDE theory.

The rest of this paper is organized as follows. Preliminaries including notations, definitions and some lemmas are presented in Section 2. Section 3 presents the main results of this paper and the proofs of the main results are given in Section 4.

\section{Preliminaries}
We begin with some notations that will be frequently used in the sequel. Let $\mathbb{R}^{n \times m}$ be the space of $(n \times m)$ real matrices equipped with the Frobenius inner product
$$
\langle M,N\rangle=tr\left(M^{\top} N\right),\quad M,N\in\mathbb{R}^{n\times m}
$$
where $M^{\top}$ denotes the transpose of $M$ and $tr\left(M^{\top} N\right)$ is the trace of $M^{\top} N$. The norm induced by the Frobenius inner product is denoted by $|\cdot|$. For a subset $\mathbb{H}$ of $\mathbb{R}^{n \times m}$, we denote by $C([0, T] ; \mathbb{H})$ the space of continuous functions from $[0, T]$ mapping into $\mathbb{H}$, $C([0, \infty) ; \mathbb{H})$ the space of continuous functions from $[0, \infty)$ mapping into $\mathbb{H}$, and by $L^{\infty}(0, T ; \mathbb{H})$ the space of Lebesgue measurable, essentially bounded functions from $[0, T]$ mapping into $\mathbb{H}$. Let $\mathbb{S}^n$ be the subspace of $\mathbb{R}^{n \times n}$ consisting of symmetric matrices and $\mathbb{S}_{+}^n$ the subset of $\mathbb{S}^n$ consisting of positive definite matrices. For $\mathbb{S}^n$-valued functions $M(\cdot)$ and $N(\cdot)$, we write $M(\cdot) \geq N(\cdot)$ (respectively, $M(\cdot)>N(\cdot)$ ) if $M(\cdot)-N(\cdot)$ is positive semidefinite (respectively, positive definite) almost everywhere with respect to the Lebesgue measure. For a positive definite matrix $M$, we denote its largest and smallest eigenvalue by $\lambda_{\max}(M)$ and $\lambda_{\min}(M)$. The identity matrix of size $n$ is denoted by $I_n$, and the zero matrix of size $n\times m$ is denoted by $0_{n\times m}$.

\subsection{Classical results of stochastic BLQ problem}
We recall some classical results of stochastic BLQ problem in the finite time horizon, as a preparation for our study. Consider the state equation
\begin{equation}\label{Y}
\left\{\begin{array}{l}
d\mathbf{Y}(t)=[\mathbf{A}(t)\mathbf{Y}(t)+\mathbf{B}(t)\mathbf{u}(t)+\mathbf{C}(t)\mathbf{Z}(t)+\mathbf{b}(t)]dt+\mathbf{Z}(t)dW_t,\quad t\in[0,T],\\
\mathbf{Y}(T)=\xi\in L^{\infty}_{\mathcal{F}_{T}}(\Omega,\mathbb{R}^{n}),
\end{array}\right.
\end{equation}
with the cost functional
\begin{equation}\label{J}
\begin{aligned}
\mathbf{J}(\xi;\mathbf{u}(\cdot))=\frac{1}{2}\mathbb{E}&\Bigg\{\langle\mathbf{G}\mathbf{Y}(0),\mathbf{Y}(0)\rangle+2\langle\mathbf{g},\mathbf{Y}(0)\rangle+\int_{0}^{T}[\langle \mathbf{Q}(t)\mathbf{Y}(t),\mathbf{Y}(t)\rangle+\langle\mathbf{N}(t)\mathbf{Z}(t),\mathbf{Z}(t)\rangle\\
&+\langle\mathbf{R}(t)\mathbf{u}(t), \mathbf{u}(t)\rangle+2\langle\mathbf{q}(t),\mathbf{Y}(t)\rangle+2\langle\mathbf{n}(t),\mathbf{Z}(t)\rangle+2\langle\mathbf{r}(t),\mathbf{u}(t)\rangle]dt\Bigg\}
\end{aligned}
\end{equation}
where in \eqref{Y}, the coefficients satisfy
$$
\mathbf{A}(\cdot),\mathbf{C}(\cdot)\in L^{\infty}\left(0,T ;\mathbb{R}^{n\times n}\right),\quad\mathbf{B}(\cdot)\in L^{\infty}\left(0,T ;\mathbb{R}^{n\times m}\right),\quad \mathbf{b}(\cdot)\in L^{\infty}\left(0,T;\mathbb{R}^{n}\right),
$$
and in \eqref{J}, the weighting coefficients satisfy
$$
\begin{aligned}
&\mathbf{G}\in\mathbb{S}^{n},\quad\mathbf{Q}(\cdot),\mathbf{N}(\cdot)\in L^{\infty}\left(0,T;\mathbb{S}^{n}\right),\quad\mathbf{R}(\cdot)\in L^{\infty}\left(0,T;\mathbb{S}^{m}\right),\\
&\mathbf{g}\in\mathbb{R}^{n},\quad\mathbf{q}(\cdot),\mathbf{n}(\cdot)\in L^{\infty}\left(0,T;\mathbb{R}^{n}\right),\quad\mathbf{r}(\cdot)\in L^{\infty}\left(0,T;\mathbb{R}^{m}\right).
\end{aligned}
$$
The standard stochastic BLQ optimal control problem on $[0,T]$ can be stated as follows.\\
\\
$\mathbf{Problem~(\mathbf{BLQ})_{[0,T]}}$. For a given terminal state $\xi\in L^{\infty}_{\mathcal{F}_{T}}(\Omega,\mathbb{R}^{n})$, find a control $\bar{\mathbf{u}}(\cdot)\in\mathscr{U}_{m}[0,T]$ such that
\begin{equation}\label{V}
\mathbf{J}(\xi;\bar{\mathbf{u}}(\cdot))=\inf_{\mathbf{u}(\cdot)\in\mathscr{U}_{m}[0,T]}\mathbf{J}(\xi;\mathbf{u}(\cdot))\equiv\mathbf{V}(\xi).
\end{equation}
The process $\bar{\mathbf{u}}(\cdot)$ (if exists) in \eqref{V} is called an (open-loop) optimal control for the terminal state $\xi$, and $\mathbf{V}(\xi)$ is called the value of Problem (BLQ)$_{[0,T]}$ at $\xi$.

The following lemma summarizes a few results for Problem (BLQ)$_{[0,T]}$. For more details, we refer to \cite{Lim and Zhou 2001,Sun and Wang 2021,Sun et al. 2021}.
\begin{lemma}\label{le-1}
Suppose that for some constant $\delta>0$ and all $t\in[0,T]$, it holds
$$
\mathbf{R}(t)\geqslant\delta I,\quad\mathbf{Q}(t),\mathbf{N}(t)\geqslant 0,\quad\mathbf{G}\geqslant 0.
$$
Then, the following hold:

(i) For every terminal state $\xi\in L^{\infty}_{\mathcal{F}_{T}}(\Omega,\mathbb{R}^{n})$, Problem (BLQ)$_{[0,T]}$ has a unique open-loop optimal control.

(ii) A pair $(\bar{\mathbf{Y}}(\cdot),\bar{\mathbf{Z}}(\cdot),\bar{\mathbf{u}}(\cdot))$ is an open-loop optimal pair of Problem (BLQ)$_{[0,T]}$ with the terminal state $\xi$ if and only if there exists an adapted process $\bar{\mathbf{X}}(\cdot)$ such that
$$
\left\{\begin{array}{l}
d\bar{\mathbf{Y}}(t)=[\mathbf{A}(t)\bar{\mathbf{Y}}(t)+\mathbf{B}(t)\bar{\mathbf{u}}(t)+\mathbf{C}(t)\bar{\mathbf{Z}}(t)+\mathbf{b}(t)]dt+\bar{\mathbf{Z}}(t)dW_t,\\
d\bar{\mathbf{X}}(t)=\left[-\mathbf{A}(t)^{\top}\bar{\mathbf{X}}(t)+\mathbf{Q}(t)\bar{\mathbf{Y}}(t)+\mathbf{q}(t)\right]dt
+\left[-\mathbf{C}(t)^{\top}\bar{\mathbf{X}}(t)+\mathbf{N}(t)\bar{\mathbf{Z}}(t)+\mathbf{n}(t)\right]dW_t,\quad t\in[0,T],\\
\bar{\mathbf{Y}}(T)=\xi,\quad\bar{\mathbf{X}}(0)=\mathbf{G}\bar{\mathbf{Y}}(0)+\mathbf{g},
\end{array}\right.
$$
and the following condition holds:
$$
-\mathbf{B}(t)^{\top}\bar{\mathbf{X}}(t)+\mathbf{R}(t)\bar{\mathbf{u}}(t)+\mathbf{r}(t)=0, \quad a.e.~t\in[0,T],~a.s.
$$

(iii) The differential Riccati equation
$$
\left\{\begin{array}{l}
\dot{\mathbf{\Sigma}}(t)-\mathbf{\Sigma}(t)\mathbf{A}(t)^{\top}-\mathbf{A}(t)\mathbf{\Sigma}(t)-\mathbf{\Sigma}(t)\mathbf{Q}(t)\mathbf{\Sigma}(t)+\mathbf{B}(t)\mathbf{R}(t)^{-1}\mathbf{B}(t)^{\top}\\
\quad+\mathbf{C}(t)\left(I_{n}+\mathbf{\Sigma}(t)\mathbf{N}(t)\right)^{-1}\mathbf{\Sigma}(t)\mathbf{C}(t)^{\top}=0,\quad t\in[0,T],\\
\mathbf{\Sigma}(T)=0
\end{array}\right.
$$
admits a unique positive semidefinite solution $\mathbf{\Sigma}(\cdot)\in C\left([0,T];\mathbb{S}^{n}\right)$. In particular, if
$$
\mathbf{Q}(t),\mathbf{N}(t)>0,
$$
then $\mathbf{\Sigma}(t)>0$ for all $t\in[0,T)$.

(iv) The unique open-loop optimal pair $(\bar{\mathbf{Y}}(\cdot),\bar{\mathbf{Z}}(\cdot),\bar{\mathbf{u}}(\cdot))$ for the terminal state $\xi$ is given by
$$
\left\{\begin{array}{l}
\bar{\mathbf{Y}}(t)=-\mathbf{\Sigma}(t)\bar{\mathbf{X}}(t)-\boldsymbol{\varphi}(t),\\
\bar{\mathbf{Z}}(t)=\left(I_{n}+\mathbf{\Sigma}(t)\mathbf{N}(t)\right)^{-1}\left(\mathbf{\Sigma}(t)\mathbf{C}(t)^{\top}\bar{\mathbf{X}}(t)-\mathbf{\Sigma}(t)\mathbf{n}(t)-\beta(t)\right),\\
\bar{\mathbf{u}}(t)=\mathbf{R}(t)^{-1}\left(\mathbf{B}(t)^{\top}\bar{\mathbf{X}}(t)-\mathbf{r}(t)\right),
\end{array}\right.
$$
where $(\boldsymbol{\varphi},\beta)$ is the solution of the following BSDE
$$
\left\{\begin{array}{l}
\begin{aligned}
d\boldsymbol{\varphi}(t)=&\left[\left(\mathbf{A}(t)+\mathbf{\Sigma}(t)\mathbf{Q}(t)\right)\boldsymbol{\varphi}(t)+\mathbf{C}(t)\left(I_{n}+\mathbf{\Sigma}(t)\mathbf{N}(t)\right)^{-1}\beta(t)-\mathbf{\Sigma}(t)\mathbf{q}(t)-\mathbf{b}(t)\right.\\
&\left.+\mathbf{C}(t)\left(I_{n}+\mathbf{\Sigma}(t)\mathbf{N}(t)\right)^{-1}\mathbf{\Sigma}(t)\mathbf{n}(t)+\mathbf{B}(t)\mathbf{R}(t)^{-1}\mathbf{r}(t)\right]dt+\beta(t)dW_{t},\quad t\in[0,T],
\end{aligned}\\
\boldsymbol{\varphi}(T)=-\xi.
\end{array}\right.
$$

(v) The value function is given by
$$
\begin{aligned}
\mathbf{V}(\xi)=&\frac{1}{2}\mathbb{E}\left[\langle \left(I_{n}+\mathbf{G}\mathbf{\Sigma}(0)\right)^{-1}\mathbf{G}\boldsymbol{\varphi}(0),\boldsymbol{\varphi}(0)\rangle-2\langle \mathbf{g},\left(I_{n}+\mathbf{\Sigma}(0)\mathbf{G}\right)^{-1}\boldsymbol{\varphi}(0)\rangle-\langle\left(I_{n}+\mathbf{\Sigma}(0)\mathbf{G}\right)^{-1}\mathbf{\Sigma}(0)\mathbf{g},\mathbf{g}\rangle\right.\\
&~\left.+\int_{0}^{T}\left(\langle\mathbf{Q}(t)\boldsymbol{\varphi}(t),\boldsymbol{\varphi}(t)\rangle-2\langle\mathbf{q},\boldsymbol{\varphi}(t)\rangle
-\langle\left(I_{n}+\mathbf{\Sigma}(t)\mathbf{N}(t)\right)^{-1}\mathbf{\Sigma}(t)\mathbf{n},\mathbf{n}\rangle-\langle\mathbf{R}(t)^{-1}\mathbf{r},\mathbf{r}\rangle\right.\right.\\
&\qquad+\left.\left.\langle\left(I_{n}+\mathbf{N}(t)\mathbf{\Sigma}(t)\right)^{-1}\mathbf{N}(t)\beta(t),\beta(t)\rangle-\langle \mathbf{n},\left(I_{n}+\mathbf{\Sigma}(t)\mathbf{N}(t)\right)^{-1}\beta(t)\rangle\right)dt\right].
\end{aligned}
$$
\end{lemma}

\subsection{Stability of solution for SDE}
 In this subsection, we collect some stability notions and results of solution for SDE. For constant matrices $\mathcal{A},\mathcal{C}\in\mathbb{R}^{n\times n},~\mathcal{B},\mathcal{D}\in\mathbb{R}^{n\times m}$, let us denote by $[\mathcal{A},\mathcal{C}]$ the following linear homogeneous uncontrolled SDE:
$$
d\mathcal{X}(t)=\mathcal{A}\mathcal{X}(t)dt+\mathcal{C}\mathcal{X}(t)dW_t,\quad t\geqslant0.
$$
For any $x\in \mathbb{R}^{n}$, there exists a unique solution $\mathcal{X}(\cdot)\equiv\mathcal{X}(\cdot;x)$ of the above satisfying $\mathcal{X}(0;x)=x$. We recall the following classical notions, which can be seen in \cite{Sun et al. 2022}.
\begin{definition}
System $[\mathcal{A},\mathcal{C}]$ is said to be

(i) $L^{2}$-stable if
$$
\mathbb{E}\int_{0}^{\infty}|\mathcal{X}(t;x)|^{2}dt<\infty,\quad\forall x\in\mathbb{R}^{n}.
$$

(ii) mean-square exponentially stable if there exists a constant $\beta>0$ such that
$$
\sup _{t\in[0,\infty)}e^{\beta t}|\mathcal{X}(t;x)|<\infty,\quad \forall x\in\mathbb{R}^{n}.
$$
\end{definition}
To characterize the above stability notions, we let $\Phi(\cdot)$ be the solution to the matrix SDE:
$$
\left\{\begin{array}{l}
d\Phi(t)=\mathcal{A}\Phi(t)dt+\mathcal{C}\Phi(t)dW_{t},\quad t\geqslant0,\\
\Phi(0)=I_{n}
\end{array}\right.
$$
We recall the following result from \cite{Sun et al. 2022}, whose proof can be seen in \cite{Huang et al. 2015,Rami and Zhou 2000,Sun and Yong 2020}.
\begin{lemma}\label{le-2}
The following are equivalent:

(i) The system $[\mathcal{A},\mathcal{C}]$ is mean-square exponentially stable.

(ii) There exist constants $\alpha,\beta>0$ such that
\begin{equation}\label{alpha,beta}
\mathbb{E}|\Phi(t)|^{2}\leqslant\alpha e^{-\beta t},\quad\forall t\geqslant0.
\end{equation}

(iii) It holds that
$$
\mathbb{E}\int_{0}^{\infty}|\Phi(t)|^{2}dt<\infty.
$$

(iv) The system $[\mathcal{A},\mathcal{C}]$ is $L^{2}$-stable.

(v) There exists a constant matrix $P\in\mathbb{S}_{+}^{n}$ such that
$$
P\mathcal{A}+\mathcal{A}^{\top}P+\mathcal{C}^{\top}P\mathcal{C}<0.
$$
\end{lemma}
We further denote by $[\mathcal{A},\mathcal{C};\mathcal{B},\mathcal{D}]$ the following controlled (homogeneous) linear system:
$$
\left\{\begin{array}{l}
d\mathcal{X}(t)=[\mathcal{A}\mathcal{X}(t)+\mathcal{B}v(t)]dt+[\mathcal{C}\mathcal{X}(t)+\mathcal{D}v(t)]dW_t,\quad t\geqslant0,\\
\mathcal{X}(0)=x,
\end{array}\right.
$$
where $v(\cdot)$ is taken from the following set of admissible controls
$$
\mathscr{U}_{m}[0,\infty)=\left\{u:[0,\infty)\times\Omega\rightarrow\mathbb{R}^{m}|u(\cdot) \text { is } \mathbb{F} \text {-progressively measurable, } \mathbb{E} \int_{0}^{\infty}|u(t)|^{2} d t<\infty\right\} .
$$
The following notion can be seen in \cite{Huang et al. 2015,Sun et al. 2022,Sun and Yong 2020}.
\begin{definition}
The system $[\mathcal{A},\mathcal{C};\mathcal{B},\mathcal{D}]$ is said to be $L^{2}$-stabilizable if there exists a matrix $\Theta\in\mathbb{R}^{m\times n}$ such that the (closed-loop) system $[\mathcal{A}+\mathcal{B}\Theta,\mathcal{C}+\mathcal{D}\Theta]$ is $L^{2}$-stable. In this case, $\Theta$ is called a stabilizer of $[\mathcal{A},\mathcal{C};\mathcal{B},\mathcal{D}]$.
\end{definition}

\section{Main results}
In this section, we present our main results. In particular, turnpike properties for Problem (BLQ)$_{T}$ are stated. First of all, we introduce the following hypotheses.\\
\\
$\mathbf{(H1)}$ The system $[A,0_{n\times n};(B~C),(0_{n\times m}~I_{n})]$ is $L^{2}$-stabilizable.\\
\\
$\mathbf{(H2)}$ The weighting matrices $Q,N\in\mathbb{S}_{+}^{n}$ and $R\in\mathbb{S}_{+}^{m}$.\\
\\
We consider the following two AREs:
\begin{equation}\label{P}
PA+A^{\top}P+Q-PBR^{-1}B^{\top}P-PC(N+P)^{-1}C^{\top}P=0,
\end{equation}
\begin{equation}\label{Sigma}
\Sigma A^{\top}+A\Sigma+\Sigma Q\Sigma-BR^{-1}B^{\top}-C(I_{n}+\Sigma N)^{-1}\Sigma C^{\top}=0.
\end{equation}
and obtain the solvability of \eqref{P} and \eqref{Sigma} in the next lemma.
\begin{lemma}\label{le-4}
Let (H1)-(H2) be satisfied, then the ARE \eqref{P} admits a unique solution $P\in\mathbb{S}_{+}^{n}$ and ARE \eqref{Sigma} admits a unique solution $\Sigma\in\mathbb{S}_{+}^{n}$.
\end{lemma}
\begin{proof}
We rewrite ARE \eqref{P} as
\begin{equation}\label{P'}
PA+A^{\top}P+Q-P(B~C)\left[\begin{pmatrix}
	R &  \\
	  & N
\end{pmatrix}+\begin{pmatrix}
0_{m\times n}\\
I_{n}
\end{pmatrix}
P
\begin{pmatrix}
	0_{n\times m} & I_{n}
\end{pmatrix}
\right]^{-1}(B~C)^{\top}P=0.
\end{equation}
Thus, under (H1)-(H2), it follows from \cite[Propsition 3.5]{Sun et al. 2022} that ARE \eqref{P} has a unique solution $P\in\mathbb{S}_{+}^{n}$. Letting $\Sigma=P^{-1}>0$, \eqref{P} is equivalent to 
\begin{equation}\label{Sigma'}
\Sigma^{-1}A+A^{\top}\Sigma^{-1}+Q-\Sigma^{-1}BR^{-1}B^{\top}\Sigma^{-1}-\Sigma^{-1}C(N+\Sigma^{-1})^{-1}C^{\top}\Sigma^{-1}=0,
\end{equation}
which yields ARE \eqref{Sigma} by multiplying first left and then right by $\Sigma$. Thus $\Sigma=P^{-1}$ is a solution of ARE \eqref{Sigma}. Conversely, let $\Sigma$ be a solution of ARE \eqref{Sigma} such that $\Sigma\in\mathbb{S}_{+}^{n}$ and $P=\Sigma^{-1}$. Similarly, by multiplying first left and then right by $\Sigma^{-1}$ in \eqref{Sigma}, we verify that $P$ is a solution of ARE \eqref{P} and its uniqueness leads to the uniqueness of $\Sigma$.
\end{proof}

The following lemma provides another perspective of hypotheses (H1)-(H2), which is useful for our study.
\begin{lemma}\label{th-1}
Let (H1)-(H2) be satisfied and set
\begin{equation}\label{A,C_Sigma}
A_{\Sigma}=-(A+\Sigma Q)^{\top},\quad C_{\Sigma}=-[C(I_{n}+\Sigma N)^{-1}]^{\top}
\end{equation}
where $\Sigma$ is the solution of ARE \eqref{Sigma}, then the system $[A_{\Sigma},C_{\Sigma}]$ is $L^{2}$-stable.
\end{lemma}
\begin{proof}
From \eqref{Sigma}, we deduce that
$$
\begin{aligned}
&\Sigma A_{\Sigma}+A_{\Sigma}^{\top}\Sigma+C_{\Sigma}^{\top}\Sigma C_{\Sigma}\\
&=-\Sigma(A+\Sigma Q)^{\top}-(A+\Sigma Q)\Sigma+C(I_{n}+\Sigma N)^{-1}\Sigma(I_{n}+N\Sigma)^{-1}C^{\top}\\
&=-\Sigma Q\Sigma-BR^{-1}B^{\top}+C(I_{n}+\Sigma N)^{-1}\Sigma(I_{n}+N\Sigma)^{-1}C^{\top}-C(I_{n}+\Sigma N)^{-1}\Sigma C^{\top}\\
&=-\Sigma Q\Sigma-BR^{-1}B^{\top}-C(I_{n}+\Sigma N)^{-1}\Sigma N\Sigma (I_{n}+N\Sigma)^{-1}C^{\top}.
\end{aligned}
$$
Since
$$
\Sigma,Q,N,R>0,
$$
then
$$
\Sigma Q\Sigma>0,~BR^{-1}B^{\top}\geqslant 0,~\Sigma N\Sigma>0.
$$
Thus,
\begin{equation*}
C(I_{n}+\Sigma N)^{-1}\Sigma N\Sigma (I_{n}+N\Sigma)^{-1}C^{\top}\geqslant 0.
\end{equation*}
Therefore
$$
\Sigma A_{\Sigma}+A_{\Sigma}^{\top}\Sigma+C_{\Sigma}^{\top}\Sigma C_{\Sigma}<0,
$$
which, combining with Lemma \ref{le-2}, completes the proof.
\end{proof}

Inspired by \cite{Sun et al. 2022}, we will introduce the static optimization problem related to Problem (BLQ)$_{T}$. By means of the positive definite matrix $\Sigma$, we can define
$$
\begin{aligned}
&H=\left\{(y,z,u)\in\mathbb{R}^{n}\times\mathbb{R}^{n}\times\mathbb{R}^{m}|Ay+Bu+Cz+b=0\right\},\\
&F(y,z,u)=\frac{1}{2}[\langle Qy,y\rangle+\langle Nz,z\rangle+\langle Ru,u\rangle+2\langle q,y\rangle+2\langle \mathbf{n},z\rangle+2\langle r,u\rangle+\langle \Sigma^{-1}z,z\rangle],
\end{aligned}
$$
and consider the following static optimization problem.\\
\\
$\mathbf{Problem~(S)}$. Find $\left(y^{*},z^{*},u^{*}\right)\in H$ such that
$$
F\left(y^{*},z^{*},u^{*}\right)=\min_{(y,z,u)\in H}F(y,z,u).
$$

For Problem (S), we have the following result.
\begin{lemma}\label{th-2}
Let (H1)-(H2) hold. Then Problem (S) admits a unique solution $\left(y^{*},z^{*},u^{*}\right)\in H$ which, together with a Lagrange multiplier $\lambda^{*} \in \mathbb{R}^{n}$, is characterized by the following system of linear equations:
\begin{equation}\label{*}
\left\{\begin{array}{l}
Qy^{*}+A^{\top}\lambda^{*}+q=0,\\
(N+\Sigma^{-1})z^{*}+C^{\top}\lambda^{*}+\mathbf{n}=0,\\
R u^{*}+B^{\top} \lambda^{*}+r=0 .
\end{array}\right.
\end{equation}
\end{lemma}
\begin{proof}
Since $[A,0_{n\times n};(B~C),(0_{n\times m}~I_{n})]$ is $L^{2}$-stabilizable, so must be $[A,0_{n\times n};(B~C),0_{n\times (n+m)}]$. In fact, if $\Theta$ is a stabilizer of $[A,0_{n\times n};(B~C),(0_{n\times m}~I_{n})]$, then there exists a $P\in \mathbb{S}_{+}^{n}$ such that
$$
P[A+(B~C)\Theta]+[A+(B~C)\Theta]^{\top}P+[(0~I_{n})\Theta]^{\top}P[(0~I_{n})\Theta]<0,
$$
which yields
$$
P[A+(B~C)\Theta]+[A+(B~C)\Theta]^{\top}P<0.
$$
Thus, $[A,0_{n\times n};(B~C),0_{n\times (n+m)}]$ is $L^{2}$-stabilizable, which is equivalent to that the matrix $\left(A-\lambda I,(B~C)\right)$ is of full rank for any $\lambda\in\mathbb{C}$ with $\operatorname{Re}\lambda \geqslant0$ (see \cite{Li et al. 2010}). So, by letting $\lambda=0$, it holds that the matrix $\left(A,B,C\right)$ has rank $n$. Then the feasible set H of Problem (S) is a non-empty closed convex set. Further, since $\Sigma>0$,
$$
F(y,z,u)\geqslant\frac{1}{2}[\langle Qy,y\rangle+\langle Nz,z\rangle+\langle Ru,u\rangle+2\langle q,y\rangle+2\langle \mathbf{n},z\rangle+2\langle r,u\rangle]
$$
is coercive on $\mathbb{R}^{n}\times\mathbb{R}^{n}\times\mathbb{R}^{m}$ due to (H2). Therefore, Problem (S) admits a unique solution.

Also, since $\left(A,B,C\right)$ is of full rank, we can obtain the optimal solution $\left(y^{*},z^{*},u^{*}\right)$ by the Lagrange multiplier method (see \cite{Yong 2018}). We denote
$$
L(y,z,u,\lambda)=F(y,z,u)+\lambda^{\top}(Ay+Bu+Cz+b).
$$
Then the unique optimal solution $\left(y^{*},z^{*},u^{*}\right)$ can be given by
\begin{equation}\label{Lyzu}
\begin{aligned}
&0=L_{y}\left(y^{*},z^{*},u^{*},\lambda^{*}\right)^{\top}=Qy^{*}+q+A^{\top}\lambda^{*},\\
&0=L_{z}\left(y^{*},z^{*},u^{*},\lambda^{*}\right)^{\top}=\left(N+\Sigma^{-1}\right)z^{*}+\mathbf{n}+C^{\top}\lambda^{*},\\
&0=L_{u}\left(y^{*},z^{*},u^{*},\lambda^{*}\right)^{\top}=Ru^{*}+r+B^{\top}\lambda^{*},\\
\end{aligned}
\end{equation}
which is actually \eqref{*}. Afterwards, we rewrite \eqref{Lyzu} as
$$
\left(\begin{array}{ccc}
Q &               &  \\
  & N+\Sigma^{-1} &  \\
  &               & R
\end{array}\right)\left(\begin{array}{c}
y^{*} \\
z^{*} \\
u^{*}
\end{array}\right)+\left(\begin{array}{c}
q \\
\mathbf{n} \\
r
\end{array}\right)+\left(\begin{array}{c}
A^{\top} \\
C^{\top} \\
B^{\top}
\end{array}\right)\lambda^{*}=0.
$$
Since the coefficient matrix is invertible, there exists a unique solution $\left(y^{*},z^{*},u^{*}\right)$ and
$$
\left(\begin{array}{c}
y^{*} \\
z^{*} \\
u^{*}
\end{array}\right)=-\left(\begin{array}{ccc}
Q &               &  \\
  & N+\Sigma^{-1} &  \\
  &               & R
\end{array}\right)^{-1}\left[\left(\begin{array}{c}
q \\
\mathbf{n} \\
r
\end{array}\right)+\left(\begin{array}{c}
A^{\top} \\
C^{\top} \\
B^{\top}
\end{array}\right)\lambda^{*}\right].
$$
Recalling the equality constraint, we can obtain that
$$
-b=Ay^{*}+Cz^{*}+Bu^{*}=-(A,C,B)\left(\begin{array}{ccc}
Q &               &  \\
  & N+\Sigma^{-1} &  \\
  &               & R
\end{array}\right)^{-1}\left[\left(\begin{array}{c}
q \\
\mathbf{n} \\
r
\end{array}\right)+\left(\begin{array}{c}
A^{\top} \\
C^{\top} \\
B^{\top}
\end{array}\right)\lambda^{*}\right].
$$
Since $(A,C,B)\in \mathbb{R}^{n\times(2n+m)}$ has rank $n$, we claim that $\lambda^{*}$ is uniquely determined by the following:
$$
\begin{aligned}
\lambda^{*}=&\left[(A,C,B)\left(\begin{array}{ccc}
Q &               &  \\
  & N+\Sigma^{-1} &  \\
  &               & R
\end{array}\right)^{-1}\left(\begin{array}{c}
A^{\top} \\
C^{\top} \\
B^{\top}
\end{array}\right)\right]^{-1}\\
&\cdot\left[b-(A,C,B)\left(\begin{array}{ccc}
Q &               &  \\
  & N+\Sigma^{-1} &  \\
  &               & R
\end{array}\right)^{-1}\left(\begin{array}{c}
q \\
\mathbf{n} \\
r
\end{array}\right)\right]
\end{aligned}
$$
Hence, $\left(y^{*},z^{*},u^{*}\right)$ is uniquely determined by \eqref{*}.
\end{proof}

Let $\left(y^{*},z^{*},u^{*}\right)$ be the solution of Problem (S) and $\lambda^{*} \in \mathbb{R}^{n}$ the corresponding Lagrange multiplier in \eqref{*}. Now we present the following main results of this paper, whose proofs are given in the Section 4. The first main theorem characterizes the weak turnpike property of Problem (BLQ)$_{T}$.
\begin{theorem}\label{th-4}
Let (H1)-(H2) hold. There exist positive constants $K_{1},\mu$, independent of $T$, such that
$$
\left|\mathbb{E}\left[\bar{Y}_{T}(t)-y^*\right]\right|+\left|\mathbb{E}\left[\bar{u}_{T}(t)-u^*\right]\right|+\left|\mathbb{E}\left[\bar{X}_{T}(t)+\lambda^*\right]\right| \leqslant K_{1}\left[e^{-\mu t}+e^{-\mu(T-t)}\right], \quad \forall t\in[0,T].
$$
\end{theorem}

The second main theorem states that the value function $V_{T}(\xi)$ of Problem (BLQ)$_{T}$ is close (in the sense of time-average) to the optimal value $V$ of Problem (S) where
$$
V=\frac{1}{2}[\langle Qy^{*},y^{*}\rangle+\langle(N+\Sigma^{-1})z^{*},z^{*}\rangle+\langle Ru^{*},u^{*}\rangle+2\langle q,y^{*}\rangle+2\langle \mathbf{n},z^{*}\rangle+2\langle r,u^{*}\rangle].
$$
\begin{theorem}\label{th-5}
Let (H1)-(H2) hold. It holds for some positive constant $K_{2}$ independent of $T$ that
$$
\left|\frac{1}{T}V_{T}(\xi)-V\right|\leqslant K_{2}\left(\frac{1}{T}+\frac{1}{T^{\frac{1}{2}}}\right).
$$
\end{theorem}

Let $X^{*}_{T}(t)$ be the unique solution to the SDE:
\begin{equation}\label{SDE-2}
\left\{\begin{array}{l}
dX^{*}_{T}(t)=\left[-\left(A+\Sigma Q\right)^{\top}X^{*}_{T}(t)\right]dt+\left[-\left(C(I_{n}+\Sigma N)^{-1}\right)^{\top}X^{*}_{T}(t)-\Sigma^{-1}z^{*}\right]dW_{t},\quad t\in[0,T]\\
X^{*}_{T}(0)=0,
\end{array}\right.
\end{equation}
where $\Sigma$ is the unique solution of ARE \eqref{Sigma}. Now we define
$$
\begin{array}{ll}
\bar{X}^{*}_{T}(t)=X^{*}_{T}(t)-\lambda^{*},&\bar{u}^{*}_{T}(t)=R^{-1}B^{\top}X^{*}_{T}(t)+u^{*}\\
\bar{Y}^{*}_{T}(t)=-\Sigma X^{*}_{T}(t)+y^{*},&\bar{Z}^{*}_{T}(t)=(I+\Sigma N)^{-1}\Sigma C^{\top}X^{*}_{T}(t)+z^{*}.
\end{array}
$$
Our third main theorem establishes the strong turnpike property of Problem (BLQ)$_{T}$.

\begin{theorem}\label{th-6}
Let (H1)-(H2) hold. There exist positive constants $K_{3},\zeta$, independent of $T$, such that
$$
\mathbb{E}\left[|\bar{Y}_{T}(t)-\bar{Y}^{*}_{T}(t)|^{2}\right]+\mathbb{E}\left[|\bar{u}_{T}(t)-\bar{u}^{*}_{T}(t)|^{2}\right]+\mathbb{E}\left[|\bar{X}_{T}(t)-\bar{X}^{*}_{T}(t)|^{2}\right]\leqslant K_{3}\left[e^{-\zeta t}+e^{-\zeta(T-t)}\right], \quad \forall t\in[0,T].
$$
\end{theorem}

\section{The turnpike property}
\subsection{Long time behavior of solution of differential Riccati equation}
The main goal of this subsection is to investigate the long time behavior of the solution of the following differential Riccati equation
\begin{equation}\label{Sigma(t)}
  \left\{\begin{array}{l}
  \dot{\Sigma}(t)+\Sigma(t)A^{\top}+A\Sigma(t)+\Sigma(t)Q\Sigma(t)-BR^{-1}B^{\top}\\
  \quad-C\left(I_{n}+\Sigma(t)N\right)^{-1}\Sigma(t)C^{\top}=0,\quad t\in[0,\infty),\\
  \Sigma(0)=0.
  \end{array}\right.
  \end{equation}
For this aim, for each $T>0$, we introduce the following differential Riccati equation:
\begin{equation}\label{SigmaT}
\left\{\begin{array}{l}
\dot{\Sigma}_{T}(t)-\Sigma_{T}(t)A^{\top}-A\Sigma_{T}(t)-\Sigma_{T}(t)Q\Sigma_{T}(t)+BR^{-1}B^{\top}\\
\quad+C\left(I_{n}+\Sigma_{T}(t)N\right)^{-1}\Sigma_{T}(t)C^{\top}=0,\quad t\in[0,T],\\
\Sigma_{T}(T)=0.
\end{array}\right.
\end{equation}
According to Lemma \ref{le-1}, we know that under (H1)-(H2), the differential Riccati equation \eqref{SigmaT} admits a unique solution $\Sigma_{T}(\cdot)\in C\left([0,T]; \mathbb{S}^{n}\right)$ satisfying $\Sigma_{T}(t)>0$ for all $0\leqslant t<T$.

In the following theorem, we show that the differential Riccati equation \eqref{Sigma(t)} is well-posed and further provide the long time behavior.
\begin{theorem}\label{th-3}
Let (H1)-(H2) be satisfied, then the differential Riccati equation $\eqref{Sigma(t)}$ admits a unique solution $\Sigma(\cdot)\in C([0,\infty);\mathbb{S}^{n})$ satisfying
$$
0<\Sigma(s)\leqslant\Sigma(t)\leqslant\Sigma,\qquad \forall~ 0<s<t<\infty,
$$
where $\Sigma$ is the unique solution of the ARE \eqref{Sigma}. Moreover, there exist positive constants $K_{4},\sigma$ such that
\begin{equation}\label{convergence}
|\Sigma-\Sigma(t)|\leqslant K_{4}e^{-2\sigma t},\quad\forall~ t\geqslant0.
\end{equation}
\end{theorem}
\begin{proof}The proof is divided into several steps.\\
$\textbf{Step 1. (Solvability of the differential Riccati equation \eqref{Sigma(t)})}$
First of all, for fixed but arbitrary $0<T_{1}<T_{2}<\infty$, we define
$$
\begin{aligned}
\Sigma^{1}(t)\triangleq \Sigma_{T_{1}}(T_{1}-t),\qquad 0\leqslant t\leqslant T_{1},\\
\Sigma^{2}(t)\triangleq \Sigma_{T_{2}}(T_{2}-t),\qquad 0\leqslant t\leqslant T_{2}.
\end{aligned}
$$
On the interval $[0,T_{1}]$, both $\Sigma^{1}$ and $\Sigma^{2}$ solve the equation 
\begin{equation*}
\left\{\begin{array}{l}
\dot{\Sigma}(t)+\Sigma(t)A^{\top}+A\Sigma(t)+\Sigma(t)Q\Sigma(t)-BR^{-1}B^{\top}-C\left(I_{n}+\Sigma(t)N\right)^{-1}\Sigma(t)C^{\top}=0,\\
\Sigma(0)=0.
\end{array}\right.
\end{equation*}
By the uniqueness, we must have $\Sigma^{1}=\Sigma^{2}$ on $[0,T_{1}]$. Then the function $\Sigma:[0,\infty)\rightarrow\mathbb{S}$ defined by
$$
\Sigma(t)\triangleq \Sigma_{T}(T-t)
$$
is independent of the choice of $T\geqslant t$ and is a solution of the differential Riccati equation \eqref{Sigma(t)}. Moreover, since for any $T\in (0,\infty)$,
$$
\Sigma_{T}(t)>0,\qquad t\in[0,T),
$$
it holds for $\forall~ t\in(0,\infty)$,
$$
\Sigma(t)>0.
$$\\
$\textbf{Step 2. (Long time behavior)}$ Taking derivative in \eqref{SigmaT} yields
\begin{equation*}
\left\{\begin{array}{l}
\ddot{\Sigma}_{T}(t)+\dot{\Sigma}_{T}(t)A_{T}(t)+A_{T}(t)^{\top}\dot{\Sigma}_{T}(t)+C_{T}(t)^{\top}\dot{\Sigma}_{T}(t)C_{T}(t)=0,\quad t\in[0,T],\\
\dot{\Sigma}_{T}(T)=-BR^{-1}B^{\top}<0.
\end{array}\right.
\end{equation*}
where
$$
\begin{aligned}
A_{T}(t)=-\left[A+\Sigma_{T}(t)Q\right]^{\top},\qquad C_{T}(t)=-[C(I_{n}+\Sigma_{T}(t)N)^{-1}]^{\top}.
\end{aligned}
$$
From \cite[Theorem 2.4]{Tang 2003}, we know that the above equation admits a unique solution $\dot{\Sigma}_{T}(t)\leqslant0$ on $[0,T]$. Following the method in Step 1, we can easily obtain $\dot{\Sigma}(t)\geqslant0$ for any $t\in(0,\infty)$, which implies 
$$0<\Sigma(s)\leqslant\Sigma(t)$$
for any $0<s<t<\infty$. Now letting $\Pi_{T}(t)=\Sigma-\Sigma_{T}(t)$, it holds that 
$$
\begin{aligned}
\dot{\Pi}_{T}(t)=&\Pi_{T}(t)A^{\top}+A\Pi_{T}(t)+\Sigma Q\Sigma-\Sigma_{T}(t)Q\Sigma_{T}(t)\\
&-C(I_{n}+\Sigma N)^{-1}\Sigma C^{\top}+C(I_{n}+\Sigma_{T}(t)N)^{-1}\Sigma_{T}(t)C^{\top}\\
=&\Pi_{T}(t)A^{\top}+A\Pi_{T}(t)+\Pi_{T}(t)Q\Pi_{T}(t)+\Pi_{T}(t)Q\Sigma_{T}(t)+\Sigma_{T}(t)Q\Pi_{T}(t)\\
&-C(I_{n}+\Sigma_{T}(t)N)^{-1}\Pi_{T}(t)(I_{n}+N\Sigma)^{-1}C^{\top}\\
=&\Pi_{T}(t)A^{\top}+\Pi_{T}(t)Q\Sigma_{T}(t)+A\Pi_{T}(t)+\Sigma_{T}(t)Q\Pi_{T}(t)+\Pi_{T}(t)Q\Pi_{T}(t)\\
&-C(I_{n}+\Sigma_{T}(t)N)^{-1}\Pi_{T}(t)(I_{n}+N\Sigma_{T}(t))^{-1}C^{\top}\\
&+C(I_{n}+\Sigma_{T}(t)N)^{-1}\Pi_{T}(t)(I_{n}+N\Sigma)^{-1}N\Pi_{T}(t)(I_{n}+N\Sigma_{T}(t))^{-1}C^{\top}\\
=&-\Bigg [\Pi_{T}(t)A_{T}(t)+A_{T}(t)^{\top}\Pi_{T}(t)+C_{T}(t)^{\top}\Pi_{T}(t)C_{T}(t)\\
&\quad-\left(\Pi_{T}(t)B_{T}+C_{T}(t)^{\top}\Pi_{T}(t)D_{T}\right)\left(R_{T}(t)+D_{T}^{\top}\Pi_{T}(t)D_{T}\right)^{-1}\left(B_{T}^{\top}\Pi_{T}(t)+D_{T}^{\top}\Pi_{T}(t)C_{T}(t)\right)\Bigg ]
\end{aligned}
$$
where
$$
\begin{aligned}
R_{T}(t)=\begin{pmatrix}
	Q^{-1} &  \\
	  & N^{-1}+\Sigma_{T}(t)
\end{pmatrix},\qquad B_{T}=\begin{pmatrix}
	I_{n} & 0_{n\times n}  
\end{pmatrix},\qquad D_{T}=\begin{pmatrix}
 0_{n\times n} & I_{n}  
\end{pmatrix}.
\end{aligned}
$$
Combining with $\Pi_{T}(T)=\Sigma>0$, it follows from \cite[Theorem 2.4]{Tang 2003} that for any $t\in[0,T]$, 
$$
\Pi_{T}(t)\geqslant0,\qquad i.e.\qquad \Sigma_{T}(t)\leqslant\Sigma.
$$
Since for any $0<s<t<\infty$, there exists $T\geqslant t$ such that 
$$
0<\Sigma_{T}(T-s)\leqslant\Sigma_{T}(T-t)\leqslant\Sigma,
$$
$$
i.e.\qquad 0<\Sigma(s)\leqslant\Sigma(t)\leqslant\Sigma,
$$
then we can deduce from the monotone convergence theorem that there exists constant matrix $\Sigma_{\infty}>0$ such that
$$
\lim\limits_{t\rightarrow\infty}\Sigma(t)=\Sigma_{\infty}\leqslant\Sigma.
$$
Now we obtain from \eqref{Sigma(t)} that for any $T>0$,
 $$
\Sigma(T+1)-\Sigma(T)=\int_{T}^{T+1}[-\Sigma(t)A^{\top}-A\Sigma(t)-\Sigma(t)Q\Sigma(t)+BR^{-1}B^{\top}+C\left(I_{n}+\Sigma(t)N\right)^{-1}\Sigma(t)C^{\top}]dt.
$$
Let $T\rightarrow\infty$ and we find that $\Sigma_{\infty}$ satisfies ARE \eqref{Sigma}, which implies
$$
\lim\limits_{t\rightarrow\infty}\Sigma(t)=\Sigma_{\infty}=\Sigma.
$$\\
$\textbf{Step 3. (Exponential rate)}$ Let $\Pi(t)=\Sigma-\Sigma(t)$. Then
$$
\begin{aligned}
\dot{\Pi}(t)=&-\Bigg[\Pi(t)A^{\top}+A\Pi(t)+\Sigma Q\Sigma-\Sigma(t)Q\Sigma(t)-C(I_{n}+\Sigma N)^{-1}\Sigma C^{\top}+C(I_{n}+\Sigma(t)N)^{-1}\Sigma(t)C^{\top}\Bigg]\\
=&-\Bigg[\Pi(t)A^{\top}+A\Pi(t)+\Sigma Q\Sigma-\big(\Pi(t)-\Sigma\big)Q\big(\Pi(t)-\Sigma\big)\\
&\quad\qquad\qquad-C\big[(I_{n}+\Sigma N)^{-1}\Sigma-(I_{n}+\Sigma(t)N)^{-1}\Sigma(t)\big]C^{\top}\Bigg]\\
=&-\Bigg[\Pi(t)A^{\top}+A\Pi(t)-\Pi(t)Q\Pi(t)+\Pi(t)Q\Sigma+\Sigma Q\Pi(t)\\
&\quad\qquad\qquad+C(I_{n}+\Sigma N)^{-1}\Pi(t)(I_{n}+N\Sigma(t))^{-1}C^{\top}\Bigg]\\
=&-\Bigg[\Pi(t)A^{\top}+A\Pi(t)-\Pi(t)Q\Pi(t)+\Pi(t)Q\Sigma+\Sigma Q\Pi(t)-C(I_{n}+\Sigma N)^{-1}\Pi(t)(I_{n}+N\Sigma)^{-1}C^{\top}\\
&\quad\qquad\qquad-C(I_{n}+\Sigma N)^{-1}\Pi(t)(I_{n}+N\Sigma(t))^{-1}N\Pi(t)(I_{n}+N\Sigma)^{-1}C^{\top}\Bigg]\\
=&-\Pi(t)(A+\Sigma Q)^{\top}-(A+\Sigma Q)\Pi(t)+C(I_{n}+\Sigma N)^{-1}\Pi(t)(I_{n}+N\Sigma)^{-1}C^{\top}\\
&\quad+\Pi(t)Q\Pi(t)+C(I_{n}+\Sigma N)^{-1}\Pi(t)(I_{n}+N\Sigma(t))^{-1}N\Pi(t)(I_{n}+N\Sigma)^{-1}C^{\top}.
\end{aligned}
$$
We recall the notation of $A_{\Sigma},C_{\Sigma}$ in \eqref{A,C_Sigma} and denote
\begin{equation}\label{f}
f(t,\Pi(t))\triangleq\Pi(t)Q\Pi(t)+C(I_{n}+\Sigma N)^{-1}\Pi(t)(I_{n}+N\Sigma(t))^{-1}N\Pi(t)(I_{n}+N\Sigma)^{-1}C^{\top}.
\end{equation}
Then we can rewrite the equation for $\Pi(\cdot)$ as follows:
$$
\dot{\Pi}(t)=\Pi(t)A_{\Sigma}+A_{\Sigma}^{\top}\Pi(t)+C_{\Sigma}^{\top}\Pi(t)C_{\Sigma}+f(t,\Pi(t)).
$$
From Lemma \ref{th-1} we know that the system $[A_{\Sigma},C_{\Sigma}]$ is $L^{2}$-stable. Thus, we conclude from \cite[Lemma 4.3]{Sun et al. 2022} and Lemma \ref{le-2} that \eqref{convergence} holds for large $t$ and hence all $t\geqslant0$.
\end{proof}

\begin{remark}\label{re-1}
As a direct consequence of Theorem \ref{th-3}, we can see
$$
|\Sigma(t)|\leqslant|\Sigma|+|\Sigma-\Sigma(t)|\leqslant |\Sigma|+K_{4},
$$
which implies $\Sigma(t)$ is uniformly bounded.
\end{remark}

Let $\left(y^{*},z^{*},u^{*}\right)$ be the unique solution of Problem (S), $\lambda^{*} \in \mathbb{R}^{n}$ be the corresponding Lagrange multiplier in Lemma \ref{th-2} and $\left(\bar{Y}_{T}(\cdot),\bar{Z}_{T}(\cdot),\bar{u}_{T}(\cdot)\right)$ be the optimal pair of Problem (BLQ)$_{T}$. Define
\begin{equation}\label{widehat}
\widehat{Y}_{T}(\cdot)=\bar{Y}_{T}(\cdot)-y^{*},\quad\widehat{Z}_{T}(\cdot)=\bar{Z}_{T}(\cdot)-z^{*},\quad\widehat{u}_{T}(\cdot)=\bar{u}_{T}(\cdot)-u^{*},\quad \widehat{X}_{T}(\cdot)=\bar{X}_{T}(\cdot)+\lambda^{*} .
\end{equation}
Then we can directly obtain from Lemma \ref{le-1} and \eqref{*} that
$$
\left\{\begin{array}{l}
d\widehat{Y}_{T}(t)=[A\widehat{Y}_{T}(t)+B\widehat{u}_{T}(t)+C\bar{Z}_{T}(t)-Cz^{*}]dt+\bar{Z}_{T}(t)dW_t,\\
d\widehat{X}_{T}(t)=\left[-A^{\top}\widehat{X}_{T}(t)+Q\widehat{Y}_{T}(t)\right]dt
+\left[-C^{\top}\widehat{X}_{T}(t)+N\bar{Z}_{T}(t)-(N+\Sigma^{-1})z^{*}\right]dW_t,\quad t\in[0,T],\\
\widehat{Y}_{T}(T)=\xi-y^{*},\quad\widehat{X}_{T}(0)=\lambda^{*},\\
-B^{\top}\widehat{X}_{T}(t)+R\widehat{u}_{T}(t)=0, \quad a.e.~t\in[0,T],~a.s.
\end{array}\right.
$$
In light of Lemma \ref{le-1}(ii), we claim that $(\widehat{Y}_{T}(\cdot),\bar{Z}_{T}(\cdot),\widehat{u}_{T}(\cdot))$ is an optimal pair of the stochastic BLQ problem with the state equation
$$
\left\{\begin{array}{l}
dY(t)=[AY(t)+Bu(t)+CZ(t)-Cz^{*}]dt+Z(t)dW_t,\quad t\in[0,T],\\
Y(T)=\xi-y^{*}
\end{array}\right.
$$
and cost functional
\begin{equation}\label{cost_1}
\begin{aligned}
J(\xi;u)=\frac{1}{2}\mathbb{E}\left\{2\left\langle\lambda^{*},Y(0)\right\rangle+\int_{0}^{T}[\right.&\langle QY(t),Y(t)\rangle+\langle NZ(t),Z(t)\rangle+\langle Ru(t),u(t)\rangle\\
& \left.\left.-2\left\langle(N+\Sigma^{-1})z^{*},Z(t)\right\rangle\right]dt\right\}.
\end{aligned}
\end{equation}

We can also obtain the following result from Lemma \ref{le-1}.
\begin{lemma}\label{le-6}
Let (H1)-(H2) be satisfied. Let $\Sigma_{T}(\cdot)$ be the solution to \eqref{SigmaT} and $(\widehat{\varphi}_{T}(\cdot),\widehat{\beta}_T(\cdot))$ be the solution to the following BSDE:
\begin{equation}\label{SDE}
\left\{\begin{array}{l}
\begin{aligned}
d\widehat{\varphi}_{T}(t)=&\left[\left(A+\Sigma_{T}(t)Q\right)\widehat{\varphi}_{T}(t)+C\left(I_{n}+\Sigma_{T}(t)N\right)^{-1}\widehat{\beta}_{T}(t)
\right.\\
&\left.+C\left(I_{n}+\Sigma_{T}(t)N\right)^{-1}(\Sigma-\Sigma_{T}(t))\Sigma^{-1}z^{*}\right]dt+\widehat{\beta}_{T}(t)dW_{t},\quad t\in[0,T],
\end{aligned}\\
\widehat{\varphi}_{T}(T)=y^{*}-\xi.
\end{array}\right.
\end{equation}
Then the optimal pair $(\widehat{Y}_{T}(\cdot),\bar{Z}_{T}(\cdot),\widehat{u}_{T}(\cdot))$ is given by
$$
\left\{\begin{array}{l}
\widehat{Y}_{T}(t)=-\Sigma_{T}(t)\widehat{X}_{T}(t)-\widehat{\varphi}_{T}(t),\\
\bar{Z}_{T}(t)=\left(I_{n}+\Sigma_{T}(t)N\right)^{-1}\left[\Sigma_{T}(t)C^{\top}\widehat{X}_{T}(t)+\Sigma_{T}(t)(N+\Sigma^{-1})z^{*}-\widehat{\beta}_{T}(t)\right],\\
\widehat{u}_{T}(t)=R^{-1}B^{\top}\widehat{X}_{T}(t).
\end{array}\right.
$$
\end{lemma}

Now we will study the property of $(\widehat{\varphi}_{T}(\cdot),\widehat{\beta}_{T}(\cdot))$. 
\begin{lemma}\label{le-7}Let (H1)-(H2) hold. The solution $(\widehat{\varphi}_{T}(\cdot), \widehat{\beta}_{T}(\cdot))$ to BSDE \eqref{SDE} satisfies
\begin{equation}\label{Ktheta}
\left|\widehat{\varphi}_{T}(t)\right|^{2}\leqslant K_{5}e^{-2\theta(T-t)},\quad\mathbb{E}\left[\int_{0}^{T}|\widehat{\beta}_{T}(t)|^{2}dt\right]\leqslant K_{5},\quad\forall t\in[0,T],
\end{equation}
for some constants $K_{5},\theta>0$ independent of $T$.
\end{lemma}
\begin{proof}
Let $\Sigma$ be the solution of ARE \eqref{Sigma}. We recall 
$$
A_{\Sigma}=-(A+\Sigma Q)^{\top},\quad C_{\Sigma}=-[C(I_{n}+\Sigma N)^{-1}]^{\top}
$$
and rewrite \eqref{SDE} as:
\begin{equation}\label{SDE-1}
\left\{\begin{array}{l}
\begin{aligned}
d\widehat{\varphi}_{T}(t)=&-\left[A_{\Sigma}^{\top}\widehat{\varphi}_{T}(t)+C_{\Sigma}^{\top}\widehat{\beta}_{T}(t)+\rho(t)\right]dt+\widehat{\beta}_{T}(t)dW_{t},\quad t\in[0,T],
\end{aligned}\\
\widehat{\varphi}_{T}(T)=y^{*}-\xi.
\end{array}\right.
\end{equation}
where
$$
\begin{aligned}
\rho(t)=&\left(\Sigma-\Sigma_{T}(t)\right)Q\widehat{\varphi}_{T}(t)+C\left(I_{n}+\Sigma_{T}(t)N\right)^{-1}(\Sigma_{T}(t)-\Sigma)N\left(I_{n}+\Sigma N\right)^{-1}\widehat{\beta}_{T}(t)\\
&\qquad\quad+C\left(I_{n}+\Sigma_{T}(t)N\right)^{-1}(\Sigma_{T}(t)-\Sigma)\Sigma^{-1}z^{*}.
\end{aligned}
$$
Since 
$$
\Sigma A_{\Sigma}+A_{\Sigma}^{\top}\Sigma+C_{\Sigma}^{\top}\Sigma C_{\Sigma}<0,
$$
one can choose $\varepsilon>0$, such that
$$
\Sigma A_{\Sigma}+A_{\Sigma}^{\top}\Sigma+(1+\varepsilon)C_{\Sigma}^{\top}\Sigma C_{\Sigma}<0.
$$
Let $P$ be the solution of ARE \eqref{P}. We denote the positive definite matrix
$$
\Lambda=-[\Sigma A_{\Sigma}+A_{\Sigma}^{\top}\Sigma+(1+\varepsilon)C_{\Sigma}^{\top}\Sigma C_{\Sigma}]
$$
and have
\begin{equation}\label{PLP}
P\Lambda P=-[A_{\Sigma}P+PA_{\Sigma}^{\top}+(1+\varepsilon)PC_{\Sigma}^{\top}\Sigma C_{\Sigma}P].
\end{equation}
For $\theta>0$, applying It\^o's formula to $s\mapsto e^{2\theta(T-s)}\langle P\widehat{\varphi}_{T}(s),\widehat{\varphi}_{T}(s)\rangle$ yields
\begin{equation}\label{Ito}
\begin{aligned}
&e^{2\theta(T-t)}\langle P\widehat{\varphi}_{T}(t),\widehat{\varphi}_{T}(t)\rangle-\langle P(y^{*}-\xi),y^{*}-\xi\rangle\\
&=\int_{t}^{T}e^{2\theta(T-s)}\left[2\langle P\widehat{\varphi}_{T}(s),A_{\Sigma}^{\top}\widehat{\varphi}_{T}(s)+C_{\Sigma}^{\top}\widehat{\beta}_{T}(s)+\rho(s)\rangle-\langle P\widehat{\beta}_{T}(s),\widehat{\beta}_{T}(s)\rangle+2\theta\langle P\widehat{\varphi}_{T}(s),\widehat{\varphi}_{T}(s)\rangle\right]ds\\
&\qquad\quad-2\int_{t}^{T}e^{2\theta(T-s)}\langle P\widehat{\varphi}_{T}(s),\widehat{\beta}_{T}(s)\rangle dW_{s}.
\end{aligned}
\end{equation}
From \eqref{PLP}, we have
\begin{equation}\label{ineq-1}
\begin{aligned}
&2\langle P\widehat{\varphi}_{T}(s),A_{\Sigma}^{\top}\widehat{\varphi}_{T}(s)+C_{\Sigma}^{\top}\widehat{\beta}_{T}(s)\rangle\\
&=\langle(A_{\Sigma}P+PA_{\Sigma}^{\top})\widehat{\varphi}_{T}(s),\widehat{\varphi}_{T}(s)\rangle+2\langle C_{\Sigma}P\widehat{\varphi}_{T}(s),\widehat{\beta}_{T}(s)\rangle\\
&=-\langle P\Lambda P\widehat{\varphi}_{T}(s),\widehat{\varphi}_{T}(s)\rangle-(1+\varepsilon)\langle PC_{\Sigma}^{\top}\Sigma C_{\Sigma}P\widehat{\varphi}_{T}(s),\widehat{\varphi}_{T}(s)\rangle+2\langle C_{\Sigma}P\widehat{\varphi}_{T}(s),\widehat{\beta}_{T}(s)\rangle\\
&=-\langle P\Lambda P\widehat{\varphi}_{T}(s),\widehat{\varphi}_{T}(s)\rangle-(1+\varepsilon)\left\langle\Sigma\left(C_{\Sigma}P\widehat{\varphi}_{T}(s)-\frac{1}{1+\varepsilon}P\widehat{\beta}_{T}(s)\right),C_{\Sigma}P\widehat{\varphi}_{T}(s)-\frac{1}{1+\varepsilon}P\widehat{\beta}_{T}(s)\right\rangle\\
&\qquad\quad+\frac{1}{1+\varepsilon}\langle P\widehat{\beta}_{T}(s),\widehat{\beta}_{T}(s)\rangle\\
&\leqslant-\langle P\Lambda P\widehat{\varphi}_{T}(s),\widehat{\varphi}_{T}(s)\rangle+\frac{1}{1+\varepsilon}\langle P\widehat{\beta}_{T}(s),\widehat{\beta}_{T}(s)\rangle.
\end{aligned}
\end{equation}
On the other hand, we have
\begin{equation}\label{ineq-2}
\begin{aligned}
2\langle P\widehat{\varphi}_{T}(s),\rho(s)\rangle=&2\langle P\widehat{\varphi}_{T}(s),\left(\Sigma-\Sigma_{T}(s)\right)Q\widehat{\varphi}_{T}(s)\rangle\\
&\qquad\quad+2\langle P\widehat{\varphi}_{T}(s),C\left(I_{n}+\Sigma_{T}(s)N\right)^{-1}(\Sigma_{T}(s)-\Sigma)N\left(I_{n}+\Sigma N\right)^{-1}\widehat{\beta}_{T}(s)\rangle\\
&\qquad\quad+2\langle P\widehat{\varphi}_{T}(s),C\left(I_{n}+\Sigma_{T}(s)N\right)^{-1}(\Sigma_{T}(s)-\Sigma)\Sigma^{-1}z^{*}\rangle\\
\leqslant&2\langle Q\left(\Sigma-\Sigma_{T}(s)\right)P\widehat{\varphi}_{T}(s),\widehat{\varphi}_{T}(s)\rangle+\frac{1}{\gamma}|\left(\Sigma_{T}(s)-\Sigma\right)\left(I_{n}+N\Sigma_{T}(s)\right)^{-1}C^{\top}P\widehat{\varphi}_{T}(s)|^{2}\\
&\qquad\quad+\gamma |N\left(I_{n}+\Sigma N\right)^{-1}\widehat{\beta}_{T}(s)|^2+\frac{1}{2}\langle P\Lambda P\widehat{\varphi}_{T}(s),\widehat{\varphi}_{T}(s)\rangle\\
&\qquad\quad+2|\Lambda^{-\frac{1}{2}}C\left(I_{n}+\Sigma_{T}(s)N\right)^{-1}(\Sigma_{T}(s)-\Sigma)\Sigma^{-1}z^{*}|^{2}\\
\leqslant&k_{51}|\Sigma_{T}(s)-\Sigma||\widehat{\varphi}_{T}(s)|^{2}+\frac{k_{51}}{\gamma}|\Sigma_{T}(s)-\Sigma|^{2}|\widehat{\varphi}_{T}(s)|^{2}+\gamma |N\left(I_{n}+\Sigma N\right)^{-1}\widehat{\beta}_{T}(s)|^2\\
&\qquad\quad+\frac{1}{2}\langle P\Lambda P\widehat{\varphi}_{T}(s),\widehat{\varphi}_{T}(s)\rangle+k_{51}|\Sigma_{T}(s)-\Sigma|^{2}
\end{aligned}
\end{equation}
where $k_{51}>0$ is a constant independent of $T$ and $\gamma>0$ is a constant to be determined later. With $\sigma$ given in Theorem \ref{convergence}, we choose $\theta=\frac{\lambda_{\min}(P\Lambda P)}{4\lambda_{\max}(P)}\bigwedge\sigma$ and some constant $\gamma>0$ such that
$$
\gamma |N\left(I_{n}+\Sigma N\right)^{-1}\widehat{\beta}_{T}(s)|^2\leqslant\frac{\varepsilon}{1+\varepsilon}\langle P\widehat{\beta}_{T}(s),\widehat{\beta}_{T}(s)\rangle.
$$
Substituting \eqref{ineq-1} and \eqref{ineq-2} into \eqref{Ito} gives
$$
\begin{aligned}
&e^{2\theta(T-t)}\langle P\widehat{\varphi}_{T}(t),\widehat{\varphi}_{T}(t)\rangle-\langle P(y^{*}-\xi),y^{*}-\xi\rangle\\
&\leqslant\int_{t}^{T}e^{2\theta(T-s)}\left[2\theta\langle P\widehat{\varphi}_{T}(s),\widehat{\varphi}_{T}(s)\rangle-\frac{1}{2}\langle P\Lambda P\widehat{\varphi}_{T}(s),\widehat{\varphi}_{T}(s)\rangle+\frac{k_{51}}{\gamma}|\Sigma_{T}(s)-\Sigma|^{2}|\widehat{\varphi}_{T}(s)|^{2}\right.\\
&\quad+\left.k_{51}|\Sigma_{T}(s)-\Sigma||\widehat{\varphi}_{T}(s)|^{2}+k_{51}|\Sigma_{T}(s)-\Sigma|^{2}\right]ds-2\int_{t}^{T}e^{2\theta(T-s)}\langle P\widehat{\varphi}_{T}(s),\widehat{\beta}_{T}(s)\rangle dW_{s}.
\end{aligned}
$$
For $0\leqslant r\leqslant t\leqslant T$, by taking the conditional expectation and recalling Theorem \ref{th-3}, we can derive the following estimate:
$$
\begin{aligned}
&\lambda_{\min}(P)\cdot\mathbb{E}_{r}[e^{2\theta(T-t)}|\widehat{\varphi}_{T}(t)|^{2}]\\
&\leqslant \lambda_{\max}(P)\cdot\mathbb{E}_{r}[|y^{*}-\xi|^{2}]+k_{51}K_{4}^{2}\int_{t}^{T}e^{(2\theta-4\sigma)(T-s)}ds\\
&\qquad+k_{51}K_{4}\mathbb{E}_{r}[\int_{t}^{T}(\frac{K_{4}}{\gamma}e^{(2\theta-4\sigma)(T-s)}+e^{(2\theta-2\sigma)(T-s)})|\widehat{\varphi}_{T}(s)|^{2}ds].
\end{aligned}
$$
Therefore, there exists a positive constant $k_{52}$ such that
$$
\mathbb{E}_{r}[e^{2\theta(T-t)}|\widehat{\varphi}_{T}(t)|^{2}]\leqslant k_{52}+k_{52}\int_{t}^{T}e^{-2\sigma(T-s)}\mathbb{E}_{r}[e^{2\theta(T-s)}|\widehat{\varphi}_{T}(s)|^{2}]ds.
$$
By Gronwall's inequality, we find
$$
\mathbb{E}_{r}[e^{2\theta(T-t)}|\widehat{\varphi}_{T}(t)|^{2}]\leqslant k_{52}+\int_{t}^{T}k_{52}\cdot k_{52}e^{-2\sigma(T-s)}\cdot e^{\int_{s}^{T}k_{52}e^{-2\sigma(T-u)}du}ds\leqslant k_{52}+\frac{k_{52}^2}{2\sigma}e^{\frac{k_{52}}{2\sigma}}:=k_{53}
$$
Therefore, we let $r=t$ and finally get
$$
|\widehat{\varphi}_{T}(t)|^{2}\leqslant k_{53}e^{-2\theta(T-t)},
$$
where $k_{53}$ is independent of $T$. Also from \eqref{SDE}, we use It\^o's formula and have
\begin{equation*}
\begin{aligned}
&\left|\widehat{\varphi}_{T}(t)\right|^{2}+\int_{t}^{T}|\widehat{\beta}_{T}(s)|^{2}ds\\
&=\left|\widehat{\varphi}_{T}(T)\right|^{2}+\int_{t}^{T}2\langle \widehat{\varphi}_{T}(s),\left(A+\Sigma Q\right)^{\top}\widehat{\varphi}_{T}(s)\rangle ds-\int_{t}^{T}2\langle \widehat{\varphi}_{T}(s),\left(\Sigma_{T}(s)-\Sigma\right)Q\widehat{\varphi}_{T}(s)\rangle ds\\
&\quad-\int_{t}^{T}2\langle \widehat{\varphi}_{T}(s),C\left(I_{n}+\Sigma_{T}(s)N\right)^{-1}\widehat{\beta}_{T}(s)\rangle ds\\
&\quad-\int_{t}^{T}2\langle \widehat{\varphi}_{T}(s),C\left(I_{n}+\Sigma_{T}(s)N\right)^{-1}\left(\Sigma-\Sigma_{T}(s)\right)\Sigma^{-1}z^{*}\rangle ds-\int_{t}^{T}2\langle\widehat{\varphi}_{T}(s),\widehat{\beta}_{T}(s)\rangle dW_{s}.
\end{aligned}
\end{equation*}
Letting $t=0$ and taking expectation, there exists a positive constant $k_{54}$ independent of $T$ such that
\begin{equation*}
\begin{aligned}
&\mathbb{E}\left[\int_{0}^{T}|\widehat{\beta}_{T}(s)|^{2}ds\right]\\
&\leqslant\mathbb{E}\left[\left|\widehat{\varphi}_{T}(T)\right|^{2}\right]+k_{54}\mathbb{E}\left[\int_{0}^{T}|\widehat{\varphi}_{T}(s)|^{2}ds\right]+k_{54}\mathbb{E}\left[\int_{0}^{T}|\Sigma_{T}(s)-\Sigma||\widehat{\varphi}_{T}(s)|^{2}ds\right]\\
&\quad+\frac{1}{2}\mathbb{E}\left[\int_{0}^{T}|\widehat{\beta}_{T}(s)|^{2}ds\right]+k_{54}\mathbb{E}\left[\int_{0}^{T}|\Sigma_{T}(s)-\Sigma|^{2}ds\right]\\
&\leqslant\frac{1}{2}\mathbb{E}\left[\int_{0}^{T}|\widehat{\beta}_{T}(s)|^{2}ds\right]+k_{53}+\frac{k_{53}k_{54}K_{4}}{2\theta}+\frac{k_{53}k_{54}}{2\theta+2\sigma}+\frac{k_{54}K_{4}^{2}}{4\sigma}.
\end{aligned}
\end{equation*}
Finally, we let $K_{5}=2k_{53}+\frac{k_{53}k_{54}K_{4}}{\theta}+\frac{k_{53}k_{54}}{\theta+\sigma}+\frac{k_{54}K_{4}^{2}}{2\sigma}$ and finish our proof.
\end{proof}
\begin{remark}\label{re-2}
Compared with \cite{Sun et al. 2022}, although the coefficients in this paper are also deterministic, our adjoint equation \eqref{SDE} is indeed BSDE instead of ODE. This is critical due to the random terminal value in BLQ problem and brings additional difficulty.
\end{remark}

\subsection{The weak turnpike property}
We are now ready to prove our weak turnpike property for Problem (BLQ)$_{T}$.

\begin{proof}[Proof of Theorem \ref{th-4}]
Thanks to Lemma \ref{le-1}, we rewrite the state equation for $\widehat{X}_{T}(\cdot)$ as
$$
\left\{\begin{array}{l}
\begin{aligned}
d\widehat{X}_{T}(t)=&\left[-\left(A+\Sigma_{T}(t) Q\right)^{\top}\widehat{X}_{T}(t)-Q\widehat{\varphi}_{T}(t)\right]dt+\left[-\left(C(I_{n}+\Sigma_{T}(t)N)^{-1}\right)^{\top}\widehat{X}_{T}(t)\right.\\
&\left.-(I_{n}+\Sigma_{T}(t)N)^{-1}(N+\Sigma^{-1})z^{*}-N(I_{n}+\Sigma_{T}(t)N)^{-1}\widehat{\beta}_{T}(t)\right]dW_t,\quad t\in[0,T],
\end{aligned}\\
\widehat{X}_{T}(0)=\lambda^{*},
\end{array}\right.
$$
where $(\widehat{\varphi}_{T}(\cdot),\widehat{\beta}_{T}(\cdot))$ is the unique solution of \eqref{SDE}. Taking expectation, we get
$$
\left\{\begin{array}{l}
d\mathbb{E}[\widehat{X}_{T}(t)]=\left\{A_{\Sigma}\mathbb{E}[\widehat{X}_{T}(t)]+Q\left(\Sigma-\Sigma_{T}(t)\right)\mathbb{E}[\widehat{X}_{T}(t)]-Q\mathbb{E}[\widehat{\varphi}_{T}(t)]\right\}dt\\
\mathbb{E}[\widehat{X}_{T}(0)]=\lambda^{*}.
\end{array}\right.
$$
Noting that $\Sigma A^{\top}+A\Sigma+\Sigma Q\Sigma-BR^{-1}B^{\top}-C(I_{n}+\Sigma N)^{-1}\Sigma C^{\top}=0$,
we obtain
\begin{equation}\label{SigmaXX}
\begin{aligned}
&\langle\Sigma\mathbb{E}[\widehat{X}_{T}(t)],\mathbb{E}[\widehat{X}_{T}(t)]\rangle-\langle\Sigma\lambda^{*},\lambda^{*}\rangle\\
&=\int_{0}^{t}\left[-\langle\left(BR^{-1}B^{\top}+\Sigma Q\Sigma+C(I_{n}+\Sigma N)^{-1}\Sigma C^{\top}\right)\mathbb{E}[\widehat{X}_{T}(s)],\mathbb{E}[\widehat{X}_{T}(s)]\rangle\right.\\
&\qquad\quad\left.+2\langle\Sigma Q(\Sigma-\Sigma_{T}(t))\mathbb{E}[\widehat{X}_{T}(s)],\mathbb{E}[\widehat{X}_{T}(s)]\rangle-2\langle\Sigma Q\mathbb{E}[\widehat{\varphi}_{T}(s)],\mathbb{E}[\widehat{X}_{T}(s)]\rangle\right]ds.
\end{aligned}
\end{equation}
Since $\Sigma>0$ and $BR^{-1}B^{\top}+\Sigma Q\Sigma+C(I_{n}+\Sigma N)^{-1}\Sigma C^{\top}>0$, it follows from \eqref{SigmaXX} that
$$
\begin{aligned}
|\mathbb{E}[\widehat{X}_{T}(t)]|^{2}\leqslant& k_{11}+\int_{0}^{t}\left[-2k_{12}|\mathbb{E}[\widehat{X}_{T}(s)]|^{2}+k_{11}|\Sigma_{T}(s)-\Sigma||\mathbb{E}[\widehat{X}_{T}(s)]|^{2}\right.\\
&\qquad\quad\left.+k_{11}|\mathbb{E}[\widehat{X}_{T}(s)]|\cdot|\mathbb{E}[\widehat{\varphi}_{T}(s)]|\right]ds,
\end{aligned}
$$
for some $k_{11},k_{12}>0$ independent of $T$. Recalling
$$
\left|\Sigma_{T}(t)-\Sigma\right|\leqslant K_{4}e^{-2\sigma(T-t)},\quad\left|\widehat{\varphi}_{T}(t)\right|^{2}\leqslant K_{5}e^{-2\theta(T-t)},\quad\forall t\in[0,T],
$$
we obtain that
$$
\begin{aligned}
|\mathbb{E}[\widehat{X}_{T}(t)]|^{2}\leqslant&k_{11}+\int_{0}^{t}\left[-k_{12}|\mathbb{E}[\widehat{X}_{T}(s)]|^{2}+k_{11}|\Sigma_{T}(s)-\Sigma||\mathbb{E}[\widehat{X}_{T}(s)]|^{2}+\frac{k_{11}^{2}}{4k_{12}}|\mathbb{E}[\widehat{\varphi}_{T}(s)]|^{2}\right]ds\\
\leqslant&k_{11}+\int_{0}^{t}\left[\left(k_{11}K_{4}e^{-2\sigma(T-s)}-k_{12}\right)|\mathbb{E}[\widehat{X}_{T}(s)]|^{2}+\frac{k_{11}^{2}K_{5}}{4k_{12}}e^{-2\theta(T-s)}\right]ds\\
\leqslant&k_{13}+\int_{0}^{t}\left[\left(k_{13}e^{-2\sigma(T-s)}-k_{12}\right)|\mathbb{E}[\widehat{X}_{T}(s)]|^{2}+k_{13}e^{-2\theta(T-s)}\right]ds\\
\end{aligned}
$$
where $k_{13}=k_{11}\bigvee (k_{11}K_{4})\bigvee\frac{k_{11}^{2}K_{5}}{4k_{12}}$. For any $0\leqslant s\leqslant t\leqslant T$, we denote
$$
\begin{aligned}
\phi(s,t)\triangleq&\exp\left\{\int_{s}^{t}(k_{13}e^{-2\sigma(T-r)}-k_{12})dr\right\}\\
=&\exp\left\{\frac{k_{13}}{2\sigma}[e^{-2\sigma(T-t)}-e^{-2\sigma(T-s)}]-k_{12}(t-s)\right\}\\
\leqslant&e^{\frac{k_{13}}{2\sigma}-k_{12}(t-s)}
\end{aligned}
$$
and have
$$
\int_{0}^{t}\phi(s,t)e^{-2\theta(T-s)}ds\leqslant e^{\frac{k_{13}}{2\sigma}}\int_{0}^{t}e^{-2\theta(T-s)}ds\leqslant\frac{1}{2\theta}e^{\frac{k_{13}}{2\sigma}}e^{-2\theta(T-t)}.
$$
Hence, by Gronwall's inequality,
$$
\begin{aligned}
|\mathbb{E}[\widehat{X}_{T}(t)]|^{2}&\leqslant k_{13}\phi(0,t)+k_{13}\int_{0}^{t}\phi(s,t)e^{-2\theta(T-s)}ds\\
&\leqslant k_{13}e^{\frac{k_{13}}{2\sigma}}e^{-k_{12}t}+\frac{k_{13}}{2\theta}e^{\frac{k_{13}}{2\sigma}}e^{-2\theta(T-t)}\\
&\leqslant k_{14}(e^{-2\mu t}+e^{-2\mu(T-t)}),\quad t\in[0,T],
\end{aligned}
$$
where $k_{14}=(1\bigvee\frac{1}{2\theta})k_{13}e^{\frac{k_{13}}{2\sigma}}$ and $\mu=\frac{k_{12}}{2}\bigwedge\theta$. By means of the relationship:
$$
\widehat{u}_{T}(t)=R^{-1}B^{\top}\widehat{X}_{T}(t),\quad\widehat{Y}_{T}(t)=-\Sigma_{T}(t)\widehat{X}_{T}(t)-\widehat{\varphi}_{T}(t),\quad t\in[0,T],
$$
we can easily verify that
$$
|\mathbb{E}[\widehat{u}_{T}(t)]|^{2}\leqslant k_{15}(e^{-2\mu t}+e^{-2\mu(T-t)}),\quad|\mathbb{E}[\widehat{Y}_{T}(t)]|^{2}\leqslant k_{16}(e^{-2\mu t}+e^{-2\mu(T-t)}),\quad t\in[0,T],
$$
for some $k_{15},k_{16}>0$ independent of $T$, which completes the proof.
\end{proof}

As a byproduct of Theorem \ref{th-4}, we have the following corollary.
\begin{corollary}\label{co-2}
Let (H1)-(H2) hold. There exists a positive constant $K_{6}$ independent of $T$  such that
$$
\begin{array}{l}
\int_{0}^{T}\left|\mathbb{E}[\bar{Y}_{T}(t)]-y^{*}\right|dt\leqslant K_{6},\qquad \int_{0}^{T}\left|\mathbb{E}[\bar{u}_{T}(t)]-u^{*}\right|dt\leqslant K_{6},\qquad
\int_{0}^{T}\left|\mathbb{E}[\bar{Z}_{T}(t)]-z^{*}\right|dt\leqslant K_{6}(1+T^{\frac{1}{2}}).
\end{array}
$$
\end{corollary}
\begin{proof}
The first two inequalities follow immediately from Theorem \ref{th-4} (one of the bounds 
 is $\frac{K_{1}}{2\mu}$). For the third one, recalling Lemma \ref{le-6}, we have
$$
\begin{aligned}
\int_{0}^{T}\left|\mathbb{E}[\bar{Z}_{T}(t)]-z^{*}\right|dt\leqslant&\int_{0}^{T}\left|\left(I+\Sigma_{T}(t)N\right)^{-1}\Sigma_{T}(t)C^{\top}E[\widehat{X}_{T}(t)]\right|dt\\
&\qquad\quad+\int_{0}^{T}\left|\left(I+\Sigma_{T}(t)N\right)^{-1}\left(\Sigma_{T}(t)-\Sigma\right)\Sigma^{-1}z^{*}\right|dt\\
&\qquad\quad+\int_{0}^{T}\left|\mathbb{E}\left[\left(I+\Sigma_{T}(t)N\right)^{-1}\widehat{\beta}_{T}(t)\right]\right|dt\\
\leqslant&k_{6}K_{1}^{2}\left|\int_{0}^{T}e^{-\mu t}+e^{-\mu(T-t)}dt\right|^{2}+k_{6}K_{4}^{2}\left|\int_{0}^{T}e^{-2\sigma(T-t)}dt\right|^{2}\\
&\qquad\quad+k_{6}T^{\frac{1}{2}}\mathbb{E}\left[\int_{0}^{T}|\widehat{\beta}_{T}(t)|^{2}dt\right]^{\frac{1}{2}}\\
\leqslant&\frac{4k_{6}K_{1}^{2}}{\mu^{2}}+\frac{k_{6}K_{4}^{2}}{4\sigma^{2}}+k_{6}K_{5}^{\frac{1}{2}}T^{\frac{1}{2}},
\end{aligned}
$$
where $k_{6}$ is a positive constant independent of $T$. We let $K_{6}=\frac{K_{1}}{2\mu}\bigvee(\frac{4k_{6}K_{1}^{2}}{\mu^{2}}+\frac{k_{6}K_{4}^{2}}{4\sigma^{2}})\bigvee k_{6}K_{5}^{\frac{1}{2}}$ and finish our proof.
\end{proof}

Now we are ready to prove Theorem \ref{th-5}.

\begin{proof}[Proof of \ref{th-5}]
Recalling the cost functional \eqref{cost_1} and Lemma \ref{le-1}, we obtain that 
\begin{equation*}
\begin{aligned}
J(\xi;\widehat{u})=&\frac{1}{2}\mathbb{E}\Bigg\{2\langle\lambda^{*},\widehat{Y}_{T}(0)\rangle+\int_{0}^{T}\bigg[\langle Q\widehat{Y}_{T}(t),\widehat{Y}_{T}(t)\rangle+\langle N\bar{Z}_{T}(t),\bar{Z}_{T}(t)\rangle+\langle R\widehat{u}_{T}(t),\widehat{u}_{T}(t)\rangle\\
&\qquad-2\langle(N+\Sigma^{-1})z^{*},\bar{Z}_{T}(t)\rangle\bigg]dt\Bigg\}\\
=&\frac{1}{2}\mathbb{E}\Bigg\{-2\langle\lambda^{*},\widehat{\varphi}_{T}(0)\rangle-\langle\Sigma_{T}(0)\lambda^{*},\lambda^{*}\rangle+\int_{0}^{T}\bigg[\langle Q\widehat{\varphi}_{T}(t),\widehat{\varphi}_{T}(t)\rangle\\
&\qquad
-\langle\left(I_{n}+\Sigma_{T}(t)N\right)^{-1}\Sigma_{T}(t)(N+\Sigma^{-1})z^{*},(N+\Sigma^{-1})z^{*}\rangle\\
&\qquad+\langle\left(I_{n}+\Sigma_{T}(t)N\right)^{-1}N\widehat{\beta}_{T}(t),\widehat{\beta}_{T}(t)\rangle+\langle(N+\Sigma^{-1})z^{*},\left(I_{n}+\Sigma_{T}(t)N\right)^{-1}\widehat{\beta}_{T}(t)\rangle\bigg]dt\Bigg\}.
\end{aligned}
\end{equation*}
Therefore, we deduce from the above equation that
\begin{equation*}
\begin{aligned}
\frac{1}{T}V_{T}(\xi)=&\frac{1}{T}\mathbb{E}\Bigg\{\int_{0}^{T}\bigg[\langle Q\bar{Y}_{T}(t),\bar{Y}_{T}(t)\rangle+\langle N\bar{Z}_{T}(t),\bar{Z}_{T}(t)\rangle+\langle R\bar{u}_{T}(t),\bar{u}_{T}(t)\rangle+2\langle q,\bar{Y}_{T}(t)\rangle\\
&\quad\qquad+2\langle \mathbf{n},\bar{Z}_{T}(t)\rangle+2\langle r,\bar{u}_{T}(t)\rangle\bigg]dt\Bigg\}\\
=&V+\frac{1}{T}\mathbb{E}\Bigg\{-2\langle\lambda^{*},\bar{Y}_{T}(0)-y^{*}\rangle-2\langle\lambda^{*},\widehat{\varphi}_{T}(0)\rangle-\langle\Sigma\lambda^{*},\lambda^{*}\rangle+\langle(\Sigma-\Sigma_{T}(0))\lambda^{*},\lambda^{*}\rangle\\
&\qquad+\int_{0}^{T}\bigg[\langle Q\widehat{\varphi}_{T}(t),\widehat{\varphi}_{T}(t)\rangle+2\langle Qy^{*}+q,\bar{Y}_{T}(t)-y^{*}\rangle+2\langle(N+\Sigma^{-1})z^{*}+\mathbf{n},\bar{Z}_{T}(t)-z^{*}\rangle\\
&\qquad+2\langle Ru^{*}+r,\bar{u}_{T}(t)-u^{*}\rangle+\langle(N+\Sigma^{-1})z^{*},(I+\Sigma_{T}(t)N)^{-1}(\Sigma-\Sigma_{T}(t))\Sigma^{-1}z^{*}\rangle\\
&\qquad+\langle\left(I_{n}+\Sigma_{T}(t)N\right)^{-1}N\widehat{\beta}_{T}(t),\widehat{\beta}_{T}(t)\rangle+\langle(N+\Sigma^{-1})z^{*},\left(I_{n}+\Sigma_{T}(t)N\right)^{-1}\widehat{\beta}_{T}(t)\rangle\bigg]dt\Bigg\}.
\end{aligned}
\end{equation*}
Thus, it follows from Theorem \ref{th-3}, Lemma \ref{le-7}, and Corollary \ref{co-2} that
\begin{equation*}
\begin{aligned}
\left|\frac{1}{T}V_{T}(\xi)-V\right|\leqslant&\frac{k_{2}}{T}+\frac{k_{2}}{T}\mathbb{E}\left[\int_{0}^{T}|\widehat{\varphi}_{T}(t)|^{2}dt\right]+\frac{k_{2}}{T}\left|\int_{0}^{T}\left(\mathbb{E}[\bar{Y}_{T}(t)]-y^{*}\right)dt\right|\\
&+\frac{k_{2}}{T}\left|\int_{0}^{T}\left(\mathbb{E}[\bar{Z}_{T}(t)]-z^{*}\right)dt\right|+\frac{k_{2}}{T}\left|\int_{0}^{T}\left(\mathbb{E}[\bar{u}_{T}(t)]-u^{*}\right)dt\right|\\
&+\frac{k_{2}}{T}\mathbb{E}\left[\int_{0}^{T}|\Sigma-\Sigma_{T}(t)|dt\right]+\frac{k_{2}}{T}\mathbb{E}\left[\int_{0}^{T}|\widehat{\beta}_{T}(t)|^{2}dt\right]+\frac{k_{2}}{T^{\frac{1}{2}}}\mathbb{E}\left[\int_{0}^{T}|\widehat{\beta}_{T}(t)|^{2}dt\right]^{\frac{1}{2}}\\
\leqslant& \frac{k_{2}}{T}+\frac{k_{2}K_{5}}{2\theta T}+\frac{2k_{2}K_{6}}{T}+k_{2}K_{6}(\frac{1}{T}+\frac{1}{T^{\frac{1}{2}}})+\frac{k_{2}K_{4}}{2\sigma T}+\frac{k_{2}K_{5}}{T}+\frac{k_{2}K_{5}^{\frac{1}{2}}}{T^{\frac{1}{2}}}
\end{aligned}
\end{equation*}
where $k_{2}$ is a positive constant independent of $T$. We let 
$$
K_{2}=(k_{2}+k_{2}K_{5}+3k_{2}K_{6}+\frac{k_{2}K_{5}}{2\theta}+\frac{k_{2}K_{4}}{2\sigma})\bigvee(k_{2}K_{6}+k_{2}K_{5}^{\frac{1}{2}})
$$ and finish our proof.
\end{proof}
\subsection{The strong turnpike property}
In this subsection, we will prove strong turnpike property for Problem (BLQ)$_{T}$. We first introduce the following lemma, which gives an estimate for $X_{T}^{*}$.
\begin{lemma}\label{le-8}
Let (H1)-(H2) hold. There exists a positive constant $K_{7}$ independent of $T$, such that the solution $X_{T}^{*}(\cdot)$ to the SDE \eqref{SDE-2} satisfies
$$
\mathbb{E}|X_{T}^{*}(t)|^{2}\leqslant K_{7},\qquad \forall t\in[0,T].
$$
\end{lemma}
\begin{proof}
Recalling 
$$
A_{\Sigma}=-(A+\Sigma Q)^{\top},\quad C_{\Sigma}=-[C(I_{n}+\Sigma N)^{-1}]^{\top},
$$
applying It\^o's formula to $s\mapsto\langle\Sigma X_{T}^{*}(s),X_{T}^{*}(s)\rangle$ yields,
$$
\begin{aligned}
&\mathbb{E}\langle\Sigma X_{T}^{*}(t),X_{T}^{*}(t)\rangle\\
&=\mathbb{E}\int_{0}^{t}\bigg[\langle\left(A_{\Sigma}^{\top}\Sigma+\Sigma A_{\Sigma}+C_{\Sigma}^{\top}\Sigma C_{\Sigma}\right)X_{T}^{*}(s),X_{T}^{*}(s)\rangle+2\langle C_{\Sigma}X_{T}^{*}(s),z^{*}\rangle+\langle\Sigma^{-1}z^{*},z^{*}\rangle\bigg]dt\\
&\leqslant\mathbb{E}\int_{0}^{t}\bigg[-k_{71}|X_{T}^{*}(s)|^{2}+2|X_{T}^{*}(s)|\cdot|C_{\Sigma}^{\top}z^{*}|+\langle\Sigma^{-1}z^{*},z^{*}\rangle\bigg]ds\\
&\leqslant\mathbb{E}\int_{0}^{t}\bigg[-\frac{k_{71}}{2}|X_{T}^{*}(s)|^{2}+\frac{2}{k_{71}}|C_{\Sigma}^{\top}z^{*}|^{2}+\langle\Sigma^{-1}z^{*},z^{*}\rangle\bigg]ds
\end{aligned}
$$
where $k_{71}>0$ is the smallest eigenvalue of the positive definite matrix 
$-(A_{\Sigma}^{\top}\Sigma+\Sigma A_{\Sigma}+C_{\Sigma}^{\top}\Sigma C_{\Sigma})$, i.e.
\begin{equation}\label{k71}
k_{71}=\lambda_{min}\left(-(A_{\Sigma}^{\top}\Sigma+\Sigma A_{\Sigma}+C_{\Sigma}^{\top}\Sigma C_{\Sigma})\right)
\end{equation}
Since $\Sigma>0$, then we can obtain
$$
\mathbb{E}|X_{T}^{*}(t)|^{2}\leqslant k_{72}\int_{0}^{t}\bigg[-\frac{k_{71}}{2}\mathbb{E}|X_{T}^{*}(s)|^{2}+\frac{2}{k_{71}}|C_{\Sigma}^{\top}z^{*}|^{2}+\langle\Sigma^{-1}z^{*},z^{*}\rangle\bigg]ds
$$
where $k_{72}$ is the largest eigenvalue of $\Sigma^{-1}$, i.e. $k_{72}=\lambda_{\max}(\Sigma^{-1})$. By comparison theorem from standard ODE theory, we finally get
$$
\mathbb{E}|X_{T}^{*}(t)|^{2}\leqslant k_{72}\left[\frac{2}{k_{71}}|C_{\Sigma}^{\top}z^{*}|^{2}+\langle\Sigma^{-1}z^{*},z^{*}\rangle\right]\int_{0}^{t}e^{\frac{k_{71}k_{72}}{2}(s-t)}ds\leqslant \frac{2}{k_{71}}\left[\frac{2}{k_{71}}|C_{\Sigma}^{\top}z^{*}|^{2}+\langle\Sigma^{-1}z^{*},z^{*}\rangle\right],
$$
which completes the proof.
\end{proof}

Finally, we give the proof of Theorem \ref{th-6}.

\begin{proof}[Proof of Theorem \ref{th-6}]
Recalling Remark \ref{re-1}, we have
$$
\begin{aligned}
A_{T}(t)=-\left[A+\Sigma_{T}(t)Q\right]^{\top},\qquad C_{T}(t)=-[C(I_{n}+\Sigma_{T}(t)N)^{-1}]^{\top},\qquad t\in[0,T],
\end{aligned}
$$
are uniformly bounded. Moreover, it follows from Theorem \ref{th-3} that for any $t\in[0,T]$,
$$
\begin{aligned}
&\left|A_{T}(t)-A_{\Sigma}\right|=\left|Q\left(\Sigma_{T}(t)-\Sigma\right)\right|\leqslant k_{31}e^{-2\sigma(T-t)}\\
&\left|C_{T}(t)-C_{\Sigma}\right|=\left|\left(I_{n}+N\Sigma\right)^{-1}N\left(\Sigma_{T}(t)-\Sigma\right)\left(I_{n}+N\Sigma_{T}(t)\right)^{-1}C^{\top}\right|\leqslant k_{31}e^{-2\sigma(T-t)}
\end{aligned}
$$
for some $k_{31},\sigma>0$, independent of $T$. We now introduce the following process
$$
\mathbb{X}_{T}(t)\triangleq\widehat{X}_{T}(t)-\mathbb{E}[\widehat{X}_{T}(t)]-X^{*}_{T}(t),\quad t\in[0,T]
$$
which satisfies $\mathbb{X}_{T}(0)=0$ and
$$
\begin{aligned}
d\mathbb{X}_{T}(t)=&\bigg\{A_{T}(t)\mathbb{X}_{T}(t)+\big[A_{T}(t)-A_{\Sigma}\big]X^{*}_{T}(t)\bigg\}dt+\bigg\{C_{T}(t)\mathbb{X}_{T}(t)+\big[C_{T}(t)-C_{\Sigma}\big]X^{*}_{T}(t)\\
&+C_{T}(t)\mathbb{E}[\widehat{X}_{T}(t)]+N(I+\Sigma_{T}(t)N)^{-1}\big[\Sigma_{T}(t)-\Sigma\big]\Sigma^{-1}z^{*}-N\big(I+\Sigma_{T}(t)N\big)^{-1}\widehat{\beta}_{T}(t)\bigg\}dW_{t}
\end{aligned}
$$
for $t\in[0,T]$. Noting $\mathbb{E}[X^{*}_{T}(t)]=\mathbb{E}[\mathbb{X}_{T}(t)]=0$ for any $t\in[0,T]$, applying Ito's formula yields
\begin{equation}\label{eq-th6-1}
\begin{aligned}
\mathbb{E}\langle\Sigma\mathbb{X}_{T}(t),\mathbb{X}_{T}(t)\rangle=&\mathbb{E}\int_{0}^{t}\bigg\{\langle\big[\Sigma A_{T}(s)+A_{T}(s)^{\top}\Sigma+C_{T}(s)^{\top}\Sigma C_{T}(s)\big]\mathbb{X}_{T}(s),\mathbb{X}_{T}(s)\rangle\\
&-2\langle\Sigma C_{T}(s)\mathbb{X}_{T}(s),N\big(I+\Sigma_{T}(s)N\big)^{-1}\widehat{\beta}_{T}(s)\rangle\\
&+\langle C_{T}(s)^{\top}\Sigma C_{T}(s)\mathbb{E}[\widehat{X}_{T}(s)],\mathbb{E}[\widehat{X}_{T}(s)]\rangle\\
&+2\langle\big[\Sigma(A_{T}(s)-A_{\Sigma})+C_{T}(s)^{\top}\Sigma(C_{T}(s)-C_{\Sigma})\big]X^{*}_{T}(s),\mathbb{X}_{T}(s)\rangle+h_{T}(s)\bigg\}ds
\end{aligned}
\end{equation}
where
$$
\begin{aligned}
h_{T}(s)=&\langle\Sigma(C_{T}(s)-C_{\Sigma})X^{*}_{T}(s),(C_{T}(s)-C_{\Sigma})X^{*}_{T}(s)\rangle\\
&+2\langle\Sigma C_{T}(s)\mathbb{E}[\widehat{X}_{T}(s)],N(I+\Sigma_{T}(s)N)^{-1}\big[\Sigma_{T}(s)-\Sigma\big]\Sigma^{-1}z^{*}\rangle\\
&-2\langle\Sigma(C_{T}(s)-C_{\Sigma})X^{*}_{T}(s)+\Sigma C_{T}(s)\mathbb{E}[\widehat{X}_{T}(s)],N\big(I+\Sigma_{T}(s)N\big)^{-1}\widehat{\beta}_{T}(s)\rangle\\
&+\left|\Sigma^{\frac{1}{2}}N(I+\Sigma_{T}(s)N)^{-1}\big[\Sigma_{T}(s)-\Sigma\big]\Sigma^{-1}z^{*}-\Sigma^{\frac{1}{2}}N\big(I+\Sigma_{T}(s)N\big)^{-1}\widehat{\beta}_{T}(s)\right|^{2}.
\end{aligned}
$$
Meanwhile, we also have
\begin{equation}\label{eq-th6-2}
\begin{aligned}
\langle\Sigma\mathbb{E}[\widehat{X}_{T}(t)],\mathbb{E}[\widehat{X}_{T}(t)]\rangle=\langle\Sigma\lambda^{*},\lambda^{*}\rangle+\int_{0}^{t}&\bigg[\langle\big[\Sigma A_{T}(s)+A_{T}(s)^{\top}\Sigma\big]\mathbb{E}[\widehat{X}_{T}(s)],\mathbb{E}[\widehat{X}_{T}(s)]\rangle\\
&-2\langle\Sigma Q\mathbb{E}[\widehat{\varphi}_{T}(s)],\mathbb{E}[\widehat{X}_{T}(s)]\rangle \bigg]ds.
\end{aligned}
\end{equation}
Therefore, by combining \eqref{eq-th6-1} and \eqref{eq-th6-2}, we obtain
\begin{equation}\label{eq-th6-3}
\begin{aligned}
\mathbb{E}&\langle\Sigma\mathbb{X}_{T}(t),\mathbb{X}_{T}(t)\rangle+\langle\Sigma\mathbb{E}[\widehat{X}_{T}(t)],\mathbb{E}[\widehat{X}_{T}(t)]\rangle\\
&=\langle\Sigma\lambda^{*},\lambda^{*}\rangle+\mathbb{E}\int_{0}^{t}\bigg\{\langle\big[\Sigma A_{T}(s)+A_{T}(s)^{\top}\Sigma+C_{T}(s)^{\top}\Sigma C_{T}(s)\big]\mathbb{X}_{T}(s),\mathbb{X}_{T}(s)\rangle\\
&\quad+\langle\big[\Sigma A_{T}(s)+A_{T}(s)^{\top}\Sigma+C_{T}(s)^{\top}\Sigma C_{T}(s)\big]\mathbb{E}[\widehat{X}_{T}(s)],\mathbb{E}[\widehat{X}_{T}(s)]\rangle\\
&\quad-2\langle\Sigma C_{T}(s)\mathbb{X}_{T}(s),N\big(I+\Sigma_{T}(s)N\big)^{-1}\widehat{\beta}_{T}(s)\rangle-2\langle\Sigma Q\mathbb{E}[\widehat{\varphi}_{T}(s)],\mathbb{E}[\widehat{X}_{T}(s)]\rangle\\
&\quad+2\langle\big[\Sigma(A_{T}(s)-A_{\Sigma})+C_{T}(s)^{\top}\Sigma(C_{T}(s)-C_{\Sigma})\big]X^{*}_{T}(s),\mathbb{X}_{T}(s)\rangle+h_{T}(s)\bigg\}ds.
\end{aligned}
\end{equation}
Noting that the positive constant $k_{71}$ is the smallest eigenvalue of the positive definite matrix 
$$-(A_{\Sigma}^{\top}\Sigma+\Sigma A_{\Sigma}+C_{\Sigma}^{\top}\Sigma C_{\Sigma}),
$$
we get
$$
\begin{aligned}
\mathbb{E}&\langle\big[\Sigma A_{T}(s)+A_{T}(s)^{\top}\Sigma+C_{T}(s)^{\top}\Sigma C_{T}(s)\big]\mathbb{X}_{T}(s),\mathbb{X}_{T}(s)\rangle\\
&=\mathbb{E}\langle\big[\Sigma A_{\Sigma}+A_{\Sigma}^{\top}\Sigma+C_{\Sigma}^{\top}\Sigma C_{\Sigma}\big]\mathbb{X}_{T}(s),\mathbb{X}_{T}(s)\rangle+\mathbb{E}\langle\Sigma\left(A_{T}(s)-A_{\Sigma}\right)\mathbb{X}_{T}(s),\mathbb{X}_{T}(s)\rangle\\
&\quad+\mathbb{E}\langle\left(A_{T}(s)-A_{\Sigma}\right)^{\top}\Sigma\mathbb{X}_{T}(s),\mathbb{X}_{T}(s)\rangle+\mathbb{E}\langle C_{T}(s)^{\top}\Sigma \left(C_{T}(s)-C_{\Sigma}\right)\mathbb{X}_{T}(s),\mathbb{X}_{T}(s)\rangle\\
&\quad+\mathbb{E}\langle\left(C_{T}(s)-C_{\Sigma}\right)^{\top}\Sigma C_{\Sigma}\mathbb{X}_{T}(s),\mathbb{X}_{T}(s)\rangle\\
&\leqslant\left(-k_{71}+k_{32}e^{-2\sigma(T-s)}\right)\mathbb{E}[|\mathbb{X}_{T}(s)|^{2}]
\end{aligned}
$$
and
$$
\begin{aligned}
\langle\big[\Sigma A_{T}(s)+A_{T}(s)^{\top}\Sigma+C_{T}(s)^{\top}\Sigma C_{T}(s)\big]\mathbb{E}[\widehat{X}_{T}(s)],\mathbb{E}[\widehat{X}_{T}(s)]\rangle\leqslant\left(-k_{71}+k_{32}e^{-2\sigma(T-s)}\right)|\mathbb{E}[\widehat{X}_{T}(s)]|^{2}
\end{aligned}
$$
where $k_{32}$ is a positive constant independent of $T$ and $k_{71}$ is defined in \eqref{k71}. Furthermore, by Theorem \ref{th-4}, it holds that
$$
\langle\big[\Sigma A_{T}(s)+A_{T}(s)^{\top}\Sigma+C_{T}(s)^{\top}\Sigma C_{T}(s)\big]\mathbb{E}[\widehat{X}_{T}(s)],\mathbb{E}[\widehat{X}_{T}(s)]\rangle\\
\leqslant 4k_{32}K_{1}^{2}e^{-2\sigma(T-s)}.
$$
On the other hand, by the Cauchy-Schwarz inequality, we have
$$
\begin{aligned}
&\mathbb{E}[-2\langle\Sigma C_{T}(s)\mathbb{X}_{T}(s),N\big(I+\Sigma_{T}(s)N\big)^{-1}\widehat{\beta}_{T}(s)\rangle]\\
&\quad\leqslant\frac{k_{71}}{4}\mathbb{E}[|\mathbb{X}_{T}(s)|^{2}]+\frac{4}{k_{71}}\mathbb{E}[|C_{T}(s)^{\top}\Sigma N\big(I+\Sigma_{T}(s)N\big)^{-1}\widehat{\beta}_{T}(s)|^{2}]\\
&\quad\leqslant\frac{k_{71}}{4}\mathbb{E}[|\mathbb{X}_{T}(s)|^{2}]+k_{33}\mathbb{E}[|\widehat{\beta}_{T}(s)|^{2}],\\
&\mathbb{E}[2\langle\big[\Sigma(A_{T}(s)-A_{\Sigma})+C_{T}(s)^{\top}\Sigma(C_{T}(s)-C_{\Sigma})\big]X^{*}_{T}(s),\mathbb{X}_{T}(s)\rangle]\\
&\quad\leqslant\frac{k_{71}}{4}\mathbb{E}[|\mathbb{X}_{T}(s)|^{2}]+\frac{4}{k_{71}}\mathbb{E}[|\big(\Sigma(A_{T}(s)-A_{\Sigma})+C_{T}(s)^{\top}\Sigma(C_{T}(s)-C_{\Sigma})\big)X^{*}_{T}(s)|^{2}]\\
&\quad\leqslant\frac{k_{71}}{4}\mathbb{E}[|\mathbb{X}_{T}(s)|^{2}]+k_{34}e^{-4\sigma(T-s)},\\
&\mathbb{E}[-2\langle\Sigma Q\mathbb{E}[\widehat{\varphi}_{T}(s)],\mathbb{E}[\widehat{X}_{T}(s)]\rangle]\leqslant\mathbb{E}[|\widehat{\varphi}_{T}(s)|^{2}]+\mathbb{E}[|Q\Sigma\widehat{X}_{T}(s)|^{2}]\leqslant K_{5}e^{-2\theta(T-t)}+k_{35}\mathbb{E}[|\widehat{X}_{T}(s)|^{2}],\\
&\mathbb{E}[h_{T}(s)]\leqslant2\mathbb{E}[|\Sigma^{\frac{1}{2}}(C_{T}(s)-C_{\Sigma})X^{*}_{T}(s)|^{2}]+2|\Sigma^{\frac{1}{2}}C_{T}(s)\mathbb{E}[\widehat{X}_{T}(s)]|^{2}\\
&\qquad\qquad\qquad+2|\Sigma^{\frac{1}{2}}N(I+\Sigma_{T}(s)N)^{-1}\big[\Sigma_{T}(s)-\Sigma\big]\Sigma^{-1}z^{*}|^{2}+4\mathbb{E}[|\Sigma^{\frac{1}{2}}N\big(I+\Sigma_{T}(s)N\big)^{-1}\widehat{\beta}_{T}(s)|^{2}]\\
&\qquad\quad~~\leqslant k_{36}e^{-4\sigma(T-s)}+k_{36}\mathbb{E}[|\widehat{X}_{T}(s)|^{2}]+k_{36}\mathbb{E}[|\widehat{\beta}_{T}(s)|^{2}]
\end{aligned}
$$
for some constants $k_{33},k_{34},k_{35},k_{36}>0$ independent of $T$. Since it follows from Lemma \ref{le-7} and Theorem \ref{th-4} that
$$
\mathbb{E}[\int_{0}^{T}|\widehat{\beta}_{T}(t)|^{2}dt]\leqslant K_{5},\quad
\int_{0}^{T}\mathbb{E}[|\widehat{X}_{T}(t)|^{2}]dt\leqslant 2K_{1}^{2}\int_{0}^{T}(e^{-2\mu t}+e^{-2\mu(T-t)})dt\leqslant\frac{2K_{1}^{2}}{\mu},
$$
we obtain from \eqref{eq-th6-3} that
$$
\begin{aligned}
\mathbb{E}[|\mathbb{X}_{T}(t)|^{2}]\leqslant k_{37}+\int_{0}^{t}\left[\left(k_{32}e^{-2\sigma(T-s)}-\frac{k_{71}k_{72}}{2}\right)\mathbb{E}[|\mathbb{X}_{T}(s)|^{2}]+k_{37}e^{-2\theta(T-s)}\right]ds.
\end{aligned}
$$
where $k_{72}=\lambda_{\max}(\Sigma^{-1})$ and $k_{37}>0$ is independent of $T$. By comparison theorem from standard ODE theory, we have
$$
\mathbb{E}[|\mathbb{X}_{T}(t)|^{2}]\leqslant k_{38}[e^{-\eta t}+e^{-\eta(T-t)}]
$$
where $\eta=\frac{k_{71}k_{72}}{2}\bigwedge2\theta$ and $k_{38}>0$ is still independent of $T$. Therefore, we have
$$
\begin{aligned}
\mathbb{E}\left[|\bar{X}_{T}(t)-\bar{X}^{*}_{T}(t)|^{2}\right]&=\mathbb{E}\left[|\mathbb{X}_{T}(t)+\mathbb{E}[\widehat{X}_{T}(t)]|^{2}\right]\\
&\leqslant2\mathbb{E}\left[|\mathbb{X}_{T}(t)|^{2}\right]+2\mathbb{E}\left[|\mathbb{E}[\widehat{X}_{T}(t)]|^{2}\right]\\
&\leqslant k_{39}[e^{-\zeta t}+e^{-\zeta(T-t)}]
\end{aligned}
$$
where $\zeta=\eta\bigwedge2\mu$ and $k_{39}=2k_{38}+4K_{1}^{2}$. Finally, we get
$$
\begin{aligned}
\bar{Y}_{T}(t)-\bar{Y}^{*}_{T}(t)=&\widehat{Y}_{T}(t)+y^{*}+\Sigma X^{*}_{T}(t)-y^{*}\\
=&-\Sigma_{T}(t)\widehat{X}_{T}(t)-\widehat{\varphi}_{T}(t)+\Sigma X^{*}_{T}(t)\\
=&-\Sigma_{T}(t)\left(\widehat{X}_{T}(t)-X^{*}_{T}(t)\right)-\widehat{\varphi}_{T}(t)+\left(\Sigma-\Sigma_{T}(t)\right)X^{*}_{T}(t)\\
=&-\Sigma_{T}(t)\left(\bar{X}_{T}(t)-\bar{X}^{*}_{T}(t)\right)-\widehat{\varphi}_{T}(t)+\left(\Sigma-\Sigma_{T}(t)\right)X^{*}_{T}(t),\\
\bar{u}_{T}(t)-\bar{u}^{*}_{T}(t)=&\widehat{u}_{T}(t)+u^{*}-R^{-1}B^{\top}X^{*}_{T}(t)-u^{*}\\
=&R^{-1}B^{\top}\left(\widehat{X}_{T}(t)-X^{*}_{T}(t)\right)\\
=&R^{-1}B^{\top}\left(\bar{X}_{T}(t)-\bar{X}^{*}_{T}(t)\right)\\
\end{aligned}
$$
and finish our proof.
\end{proof}

\end{document}